\documentclass[9pt,twoside,reqno]{amsart}
\allowdisplaybreaks
\usepackage{amsmath,amstext,amssymb,epsfig}
\usepackage{graphicx}
\usepackage{mathrsfs}
\calclayout
\makeatletter
\renewcommand{\@seccntformat}[1]{\bf\csname the#1\endcsname.}
\renewcommand{\section}{\@startsection{section}{1}
	\z@{.7\linespacing\@plus\linespacing}{.5\linespacing}
	{\normalfont\upshape\bfseries\centering}}
\renewcommand{\@biblabel}[1]{\@ifnotempty{#1}{#1.}}
\makeatother
\theoremstyle{plain}
\newtheorem{thm}{Theorem}[section]
\newtheorem{lem}[thm]{Lemma}
\newtheorem{prop}[thm]{Proposition}
\newtheorem{cor}[thm]{Corollary}

\theoremstyle{definition}

\newtheorem{defn}[thm]{Definition}
\newtheorem{rem}{Remark}[section]

\usepackage{cancel}
\usepackage[parfill]{parskip}
\usepackage[german]{varioref}
\usepackage[all]{xy}
\usepackage{stmaryrd}
\usepackage{color}
\usepackage{indentfirst} 
\allowdisplaybreaks
\def\A{{\mathcal A}}

\def\T{{\mathcal T}}
\def\>{\succ}
\def\<{\prec}

\def\b{\beta}
\def\a{\alpha}
\def\l{\lambda}
\def\p{\partial}

\def\m{\mu}

\begin{document}	
	\title[Sania Asif \textsuperscript{1}, Yao Wang  \textsuperscript{2}]
	{Rota-Baxter operators and Loday-type algebras on the BiHom-associative conformal algebras}
	\author{ Sania Asif \textsuperscript{1}, Yao Wang \textsuperscript{2,*}}
	\address{\textsuperscript{1}School of Mathematics and Statistics, Nanjing University of information science and technology, Nanjing, Jianngsu Province, PR China.}
	\address{\textsuperscript{2}School of Mathematics and Statistics, Nanjing University of information science and technology, Nanjing, Jianngsu Province, PR China.}
	
	
	\email{\textsuperscript{1}11835037@zju.edu.cn, 200036@nuist.edu.cn}
	\email{\textsuperscript{2}wangyao@nuist.edu.cn}
	
	\keywords{(tri-)dendriform algebra, BiHom-associative algebra, Associative-conformal algebras, BiHom-NS-algebra, Quadri-conformal algebra}
	\subjclass[2010]{16R60, 17B05, 17B40, 17B37}
	
	\date{\today}
\thanks{This work is supported by the Jiangsu Natural Science Foundation Project (Natural Science Foundation of Jiangsu Province), Relative Gorenstein cotorsion Homology Theory and Its Applications (No.BK20181406). }
	\begin{abstract}(Tri)dendriform algebras, Rota-Baxter operators, and closely related NS-algebras have a number of dominant applications in physics, especially in quantum field theory. Proceeding from the recent study relating these structures, this paper considers (tri)dendriform algebras, NS-algebras, and (twisted)Rota-Baxter operators in the context of BiHom-associative conformal algebras. A comprehensive investigation of the BiHom-(tri)dendriform conformal algebras and their characterization in terms of conformal bimodule has been conducted. The study of BiHom-NS-conformal algebra reveals that it is not only a generalization of NS-conformal algebra using two structural maps but is also the generalization of BiHom-(tri)dendriform conformal algebras. Additionally, it is found to have a close proximity between BiHom-twisted Rota-Baxter operators and BiHom-NS-conformal algebras. The comparative study to Rota-Baxter operators on BiHom-associative conformal algebras and Rota-Baxter operators on BiHom-(tri)dendriform conformal algebras reveals a relationship between BiHom-quadri conformal algebra and Rota-Baxter operators. In the end, the concept of Rota-Baxter system (a generalization of the Rota-Baxter operator) for BiHom-associative conformal algebras and BiHom-dendriform conformal algebras is narrated, where the interconnections of these algebras are depicted. Furthermore, a connection is established between BiHom-quadri conformal algebras and Rota-Baxter systems for BiHom-dendriform conformal algebras. 
\end{abstract}
\footnote{1 Sania Asif (Email: 11835037@zju.edu.cn ), 2 \textbf{Corresponding Author}: Yao Wang (Email: wangyao@nuist.edu.cn )
}
\maketitle
\section{Introduction}\label{introduction}
Rota-Baxter operator on an associative algebra $A$ over a field $\mathbb{F}$ is a linear map $R:A\to A$ that satisfies the following identity:\begin{equation*}
R(p)R(q)  = R(R(p)q + pR(q)+\theta(pq)), \quad \forall p, q\in \A, \theta \in \mathbb{F}.
\end{equation*}The concept of a Rota-Baxter operator was first introduced by G. Baxter in \cite{B1}. Later, this idea was thoroughly investigated in a variety of mathematical domains (for more information, see \cite{A2, 4, 7}), where its relationship with combinatorics and many other fields of mathematical physics were primarily introduced. Rota-Baxter operators are not only important in the study of the integrable systems as shown in \cite{10}, but in addition, when defined on associative algebra, verify the traditional Yang-Baxter equation (\cite{12}). The relationship between Rota-Baxter operators and classical Yang-Baxter equations, (tri)dendriform algebras, free Rota-Baxter algebras, pre-Lie algebras, and quantum field theory, can be seen in (\cite{13, 17, 16, 18, 39}). There is another well known structure of dendriform algebra that was introduced by Loday in the study of $\mathcal{K}$-theory in \cite{20}. Aguiar showed that the structure of a dendriform algebra can also be obtained from Rota-Baxter operator of wight zero. Moreover, it occur as Koszul-dual to diassociative algebras. Dendriform algebra can be linked to Rota-Baxter algebra in the form of adjoint functors; see \cite{11, 14}. Later on, It is discovered that the fields of homology, operads, Gerstenhaber algebra, Hopf algebras, homotopy theory, quantum field theory, and combinatorics all have a strong relationship to the dendriform algebra. Dendriform algebra further splits into quadri algebra, which is essentially an associative algebra for which the multiplication can be decomposed as the sum of four operations in a certain coherent manner in \cite{1}. According to Leroux \cite{L} and Uchino \cite{U1}, NS-algebra (with reference to associative algebra) is an algebraic structure with three operations $\>$, $\<$, and $\vee$ that satisfy certain axioms, indicating that the new operation $*=  \>+ \<+ \vee$ is associative. In addition to tridendriform algebra, NS-algebras also generalize dendriform algebras. A generalization of Rota-Baxter operators of wieght zero that involves Hochschild $2$-cocycle is usually refer to twisted Rota-Baxter operator have a significant relation with the dendriform algebra, studied in \cite{U2}. 
\par In \cite{25}, Makhlouf proposed the concept of Hom-associative algebra, where twisted associative identity was used with the  help of a linear map $\a$. Since then, several algebraic concepts, including Hom-Hopf algebras, Hom-coalgebras, Hom-bialgebras, generalized Hom-Lie algebras, and Hom-Poisson algebras, Hom-Lie superalgebras and Hom-Hopf modules have been introduced in terms of twisted maps (for more information, see \cite{30, 27}) that have a significant contribution to understanding the structural theory of various mathematical objects. Recently, a more general structure was introduced in \cite{GMMP}, that  is called a "BiHom-algebra"(associative or non-associative) involving two structural linear maps, $\a$ and $\b$. For more details about BiHom-type algebras, readers are referred to \cite{LMMP2, LMMP3, LMMP4} and references therein. In particular, the BiHom-analogue of Rota-Baxter operators and dendriform algebras was introduced in \cite{LMMP} and the result that the dendriform algebra can be obtained from Rota-Baxter operator of weight zero, also valid for BiHom-case.
\par Conformal algebras play a significant part in the structure theory development of (associative and Lie) conformal algebras. This algebra gives a complete insight of the operator product expansion (OPE) of chiral fields in conformal field theory (CFT). The study of vertex algebras in \cite{32} introduced the concept of conformal algebras. Lie and Associative conformal algebras are interlinked because they naturally appears in the representation theory. In a series of papers (see \cite{2, 34, 35, 40, 38}), structure theory and representation theory of associative conformal algebras are studied. Conformal algebras, in particular, have strong ties to infinite-dimensional algebras that satisfy the locality property. Later on, associative conformal algebras in Hom-setting were studied in \cite{45, 44, 41}, where derivations, representations theory, and cohomology theory of Hom-type conformal algebras were discussed in detail. Furthermore, associative conformal algebras in BiHom-setting were studied in \cite{SXS, TC}, and references therein.
\par In the present paper, we study the BiHom-(tri)dendriform conformal algebra and provide its characterization in terms of conformal bimodules. We then studied BiHom-NS-conformal algebra and showed that it is not only a generalization of NS-conformal algebra using two structural maps but is also the generalization of BiHom-(tri)dendriform conformal algebra. Furthermore, in the context of bihom-associative conformal algebras, we discuss its relationship with twisted Rota-Baxter operators. We then looked how Rota-Baxter operators on BiHom-associative conformal algebras relate to BiHom-(tri)dendriform conformal algebras. We look into the Rota-Baxter operator on BiHom-dendriform conformal algebra as well as BiHom-associative conformal algebra, which leads to BiHom-quadri conformal algebra. We further study some of its important properties. In the end, we investigate  the further generalization of the Rota-Baxter operator that is called Rota-Baxter system for BiHom-associative conformal algebra and BiHom-dendriform conformal algebra. Moreover, we study the relationship of Rota-Baxter systems for BiHom-dendriform conformal algebras with the BiHom-quadri conformal algebras. This research is crucial since it investigates algebra and a variety of operators while also demonstrating how they are related. Additionally, as this study takes BiHom- associative conformal algebras into account, we can obtain equivalent conclusions for (Hom-)associative -conformal algebras. These findings also serve as a foundation for understanding conformal algebra structure theory and subsequent developments like cohomology and deformation theory.    
 \par We organize this paper as follows:
 In Section $2$, we studied dendriform and tridendriform conformal algebra on BiHom-associative conformal algebras and discussed their behavior in terms of conformal bimodules, which further led us to discuss BiHom-NS conformal algebra. In Section $3$, we focused on describing Rota-Baxter operators on BiHom-associative conformal algebras and its relation with above mentioned algebras. Afterward, we define Rota-Baxter operators on BiHom-dendriform conformal algebras that is helpful in obtaining BiHom-quadri conformal algebra structure (discussed in section $4$ in more details). In addition, we discussed twisted-Rota-Baxter operators and their relation to BiHom-NS conformal algebra.
 We defined BiHom-quadri conformal algebra, which is essentially a quadri-conformal algebra with two structural maps, and evaluated vertical and horizontal BiHom-dendriform conformal algebras from it in Section $4$. We then show that new multiplications for a given BiHom-quadri-conformal algebra yield BiHom-associative conformal algebra. In addition, we show that a quadri-conformal algebra can be obtained from the tensor product of two BiHom-dendriform conformal algebras, similar concept for BiHom-type algebra is given in  \cite{QLJ}. In the last Section $5$, we define  the "Rota-Baxter system" on BiHom-associative conformal algebra and BiHom-dendriform conformal algebras, and investigate its relation with BiHom-quadri-conformal algebras.\par All linear maps, vector spaces, and tensor products in this paper are over the field of complex numbers $\mathbb{C}$. Where $A$ stands for BiHom-associative conformal algebra with elements $p, q$ and $r.$
\section{ BiHom-(tri)dendriform conformal algebra and BiHom-NS conformal algebra}In this section, we study dendriform and tridendriform conformal algebras corresponding to BiHom-associative conformal algebras and discuss their behavior in terms of conformal bimodules, which further leads to discussing BiHom-NS conformal algebra. First of all, we recall some important definitions to support our findings.
\begin{defn}A $\mathbb{C}[\p]$-module $A$ is called BiHom-associative conformal algebra with $\mathbb{C}$-linear commuting maps $\a, \b : A \to A,$ such that $\a\p= \p\a $, $\b\p= \p\b$, if it is equipped with a $\l$- multiplication $\tau_\l$ which defines a $\mathbb{C}$-bilinear map $\tau_\l: A \otimes A\to A[\l]= \mathbb{C}[\l]\otimes A$, with notation $\tau_\l(p\otimes q)= p_\l q$, for all $p, q\in A$, satisfying the following axioms:
	\begin{eqnarray}
	(\p p)_\l q = -\l (p_\l q),& p_\l(\p q) = (\p+\l)(p_\l q),&\textit{(Conformal sesqui-linear identities)}\\
	\a(p_\l q) = \a(p)_\l \a(q)\textit{ and } &\b(p_\l q)= \b(p)_\l\b(q), &\textit{(Conformal BiHom-multiplicative identities)}\\
	&\a(p)_\l(q_\m r) = (p_\l q)_{\l+\m}\b(r).&\textit{(Conformal BiHom- associative identity)} \label{eqA4}
	\end{eqnarray} for $\l,\m\in \mathbb{C}, p,q,r \in A.$ Where $\a$ and $\b$ are called the structure maps of $\A$.\end{defn}
A morphism of BiHom-associative conformal algebras $f: (A, \tau_A, \a_A, \b_A) \to (B, \tau_B, \a_B, \b_B)$ is a $\mathbb{C}$-linear map $f : A \to B$, that satisfies $\a_B  f = f \a_A$, $\b_B f= f \b_A$ and $f \tau_A = \tau_B (f \otimes f)$.

\begin{defn}\label{def2.2}
Assume that $(A, \tau_{\l}^A, \a_A, \b_A)$ be a BiHom-associative conformal algebra and $(M, \a_M, \b_M )$ is a triple where $M$ is $\mathbb{C}[\p]$-module equipped with a $\mathbb{C}$-linear commuting maps $\a_M, \b_M \in End(M)$ in such a way that $\a_M(p_\l m)={\a_{A}(p)}_\l\a_M(m)$ and $\b_M(p_\l m)={\b_{A}(p)}_\l\b_M(m)$.
	\begin{enumerate}
		\item If there is a conformal sesquilinear map $(m, p)\mapsto m_{\l}p$, that satisfies $(m_{\l}p)_{\l+\m}\b_{A}(q) = \a_{M}(m)_{\l} (p_{\m} q)$, for all $p,q\in \A$ and $m \in M$. Then $\mathbb{C}[\p]$-module $M$ is referred to as a right conformal module over $A$.
		\item If there is a conformal sesquilinear map $(p, m)\mapsto p_{\l}m$, that satisfies ${(p_{\l}q)}_{\l+\m} \b_{M}(m)= {\a_{A}(p)}_{\l} (q_{\m}m)$, for all $p,q \in A$ and $m\in M$. then $\mathbb{C}[\p]$-module $M$ is referred to as a left conformal module over $A$.
		\item If $M$ is both a right conformal module and a left conformal module and meets the compatibility condition ${(p_{\l}m)}_{\l+ \m}\b_A(q) = {\a_A(p)}_{\l}(m_{\m} q), $ for all $p, q \in A$ and $m \in M,$ then $M$ is referred to as a conformal bimodule over $A$ (or conformal $A$-bimodule).
\end{enumerate}
\end{defn}
\begin{defn}A $5$-tuple $(A, \<_\l, \>_\l, \a, \b)$ equipping a $\mathbb{C}[\p]$-module $A$, $\mathbb{C}$-bilinear maps $\<_\l, \>_\l: A \otimes A \to A[\l]$ and commuting $\mathbb{C}$-linear maps $\a, \b : A \to A$ is said to be a BiHom-dendriform conformal algebra , if the following conditions hold:
	\begin{eqnarray}(\p p)\>_\l q&= -\l (p \>_\l q), &p \>_\l (\p q)= (\l+ \p) (p \>_\l q),\\(\p p)\<_\l q&= -\l (p\<_\l q), &p\<_\l (\p q)= (\l+ \p) (p \<_\l q),\\
	\a(p \<_\l q) &= \a(p) \<_\l \a(q),& \a(p \>_\l q) = \a(p) \>_\l \a(q),\label{D1} \\
	\b(p \<_\l q)&= \b(p) \<_\l \b(q), & \b(p \>_\l q) = \b(p) \>_\l \b(q),\label{D2}  \\
	&(p \<_\l q) \<_{\l+\m} \b(r) &= \a(p) \<_\l (q \<_\m r + q \>_\m r),\label{D3} \\
	&(p \>_\l q) \<_{\l+\m} \b(r) &= \a(p) \>_{\l} (q \<_\m r), \label{D4} \\
	&\a(p) \>_\l (q \>_\m r) &= (p \<_\l q + p \>_\l q) \>_{\l+\m} \b(r),\label{D5} \end{eqnarray}
	for all $\l, \m\in \mathbb{C}$ and $p,q,r \in A$.\end{defn}
\begin{defn}A morphism of BiHom-dendriform conformal algebras $f : (A,\<_\l,\>_\l,\a,\b) \to (A',\<'_\l,\>'_\l,\a',\b')$ is a $\mathbb{C}$-linear map $f : A \to A'$, that satisfies $f(p \<_\l q) = f(p) \<'_\l f(q)$ and $f(p \>_{\l} q) = f(p) \>'_{\l} f(q),$ for all $p,q \in A$, as well as $f \a = \a' f $ and $f \b = \b' f$.\end{defn}
\begin{prop}\label{prop3.5}Let $(A, \<_\l, \>_\l)$ be a dendriform conformal algebra, $\a, \b\in cend(A)$, conformal bilinear maps $\<_\l^{(\a,\b)}, \>_\l^{(\a,\b)}: A \otimes A \to A[\l]$ defined by $p \<_\l^{(\a,\b)} q = \a(p) \<_\l \b(q)$ and $p \>_\l^{(\a,\b)} q = \a(p) \>_\l \b(q)$ for all $p, q\in A$. Then $A_{(\a,\b)} := (A, \<_{\l}^{(\a,\b)}, \>_{\l}^{(\a,\b)},\a,\b)$ is a BiHom-dendriform conformal algebra, called the Yau-twist of $A$. Furthermore, assume that $(A', \<'_{\l}, \>'_{\l},\a',\b')$ is another BiHom-dendriform conformal algebra and satisfying $f \circ \a = \a' \circ f$ and $f \circ \b = \b ' \circ f$. Then $f : A_{(\a,\b)}\to A'_{(\a',\b')}$ is a morphism of BiHom-dendriform conformal algebras, if $f:A\to A'$ is a BiHom tridendriform-conformal algebra morphism.
\end{prop}
\begin{proof}To show that $A_{(\a,\b)}$ is a BiHom-dendriform conformal algebra, we show that Eqs. (\ref{D1})- (\ref{D5}) holds. We only show one identity and leave the rest to the readers. By using the formulae for $\<_{\l}^{(\a,\b)}$ and $\>_{\l}^{(\a,\b)}$ with the fact that $\a$ and $\b$ are commuting dendriform conformal algebra endomorphisms, we compute that \begin{equation*}
	\begin{aligned}
	(p \<_{\l}^{(\a,\b)} q) \<_{\l+\m}^{(\a,\b)} \b(r) &= (\a^2 (p) \<_{\l} \a\b(q)) \<_{\l+\m} \b^2(r)\\&=  \a^2(p)\<_{\l}(\a\b(q) \<_{\m} \b^2(r)+ (\a\b(q) \>_{\m} \b^2(r)))\\&= \a(p) \<_{\l}^{(\a,\b)} (q \<_{\m}^{(\a,\b)} r+ q \>_{\m}^{(\a,\b)}r),
	\end{aligned}
	\end{equation*} for all $p,q,r \in A$ and $\l,\m\in \mathbb{C}$.
\end{proof}
\begin{rem}
	In a broader sense, consider $(A, \<_\l, \>_\l, \a, \b)$ as a BiHom-dendriform conformal algebra and $\tilde{\a}, \tilde{\b} \in cend(A)$ as two morphisms of BiHom-dendriform conformal algebras such that any two of the maps $\a,\b , \tilde{\a}, \tilde{\b} $ commutes. For any  $p,q\in A.$, define a new multiplication on $A$ by $p \<'_\l q= \tilde{\a}(p) \<_\l\tilde{\b}(q)$ and $p \>'_\l q = \tilde{\a}(p) \>_\l \tilde{\b}(q).$  Then it can be demonstrated that $(A, \<'_\l, \>'_\l , \a\circ \tilde{\a}, {\b}\circ \tilde{\b})$ is a BiHom-dendriform conformal algebra.  
\end{rem}In term of conformal bimodule, We can characterize BiHom-dendriform conformal algebras as follows:\begin{prop} Assume that  $A$ is a linear space along with the bilinear multiplication $\<_\l, \>_\l: A \otimes A \to A[\l]$ and two commuting linear maps $\a, \b: A\to A$ that are multiplicative with regards to $\<_\l$ and $\>_\l$. Define $p *_\l q = p \<_\l q + p \>_\l q,$ for all $p,q \in A.$ Then $(A, \<_\l, \>_\l, \a, \b)$ is a BiHom-dendriform conformal algebra if and only if $(A, *_\l, \a, \b)$ is a BiHom-associative conformal algebra and $(A, \a, \b)$ is a conformal bimodule over $(A, *_\l, \a, \b)$, with actions $p_\l m = p\>_\l m$ and $m_\l p = m \<_\l p,$ for all $p, m \in A$ and $\l\in \mathbb{C}$.
\end{prop}
\begin{defn}\label{deftridendriform}A $6$-tuple  $(A, \<_\l,\>_\l,._\l, \a,\b)$ equipping a $\mathbb{C}[\p]$-module $A$, three bilinear maps $\<_\l, \>_\l, ._\l: A \otimes A \to A[\l]$ and  two commuting linear maps $\a, \b : A \to A$ is said to be a BiHom-tridendriform conformal algebra, if the following conditions hold:
	\begin{eqnarray}
	(\p p)\>_\l q= -\l (p \>_\l q), & (\p p)\<_\l q=-\l (p\<_\l q),	 (\p p)._\l q= -\l (p ._\l q), \\p \>_\l (\p q)= (\l+\p) (p \>_\l q), &p\<_\l (\p q)= (\l+\p) (p \<_\l q), p ._\l (\p q)= (\l+\p) (p ._\l q),\\
	\a(p\<_\l q) = \a(p) \<_\l \a(q),& \a(p \>_\l q) = \a(p) \>_\l \a(q), \a(p ._\l q) = \a(p) ._\l \a(q),\label{TD1}\\
	\b(p \<_\l q) = \b(p) \<_\l \b(q),& \b(p \>_\l q) = \b(p) \>_\l \b(q), \b(p._\l q) = \b(p) ._\l \b(q),\label{TD2} \\&
	(p \<_\l q) \<_{\l+\m} \b(r) = \a(p) \<_\l(q \<_\m r + q \>_\m r + q ._\m r),\label{TD3} \\&
	(p\>_\l q) \<_{\l+\m} \b(r) = \a(p) \>_\l (q \<_{\m} r), \label{TD4}\\&
	\a(p) \>_\l (q \>_\m r) = (p \<_{\l} q + p \>_\l q+ p ._\l q) \>_{\l+\m} \b(r),\label{TD5}\\& 
	\a(p) ._{\l} (q\>_\m r)= (p\<_\l q) ._{\l+\m}\b(r),\label{TD6}\\&
	\a(p)\>_\l (q._\m r) = (p\>_\l q) ._{\l+\m} \b(r),\label{TD7}\\&
	\a(p) ._\l(q \<_\m r) = (p ._\l q) \<_{\l+\m}\b(r),\label{TD8}\\&
	\a(p) ._\l(q ._\m r) = (p ._\l q) ._{\l+\m}\b (r)\label{TD9},\end{eqnarray}for all $\l,\m\in \mathbb{C}$ and $p,q,r\in A$.
\end{defn}
A BiHom-tridendriform conformal algebra $(A, \<_\l, \>_\l,._\l,\a,\b)$ yields BiHom-dendriform conformal algebra by using $p._{\l} q = 0$ for all $p, q \in A$ .
For converse case:
\begin{prop}\label{proptri1}
Consider a BiHom-tridendriform conformal algebra $(A, \<_\l, \>_\l, ._\l, \a,\b)$, We can generate a BiHom-dendriform conformal algebra $(A', \<_\l', \>_\l', \a,\b)$ out of it by  defining new multiplication by $p \>' q = p \>_\l q,$ and  $p \<'_\l q = p \<_\l q + p ._\l q$ for all $p, q \in A.$
\end{prop}

\begin{proof}
	Note that Eq. (\ref{D1}) and (\ref{D2}) for $\<'_\l$ and $\>'_\l$ are obvious. We check the remaining relations one by one:
	\begin{equation*}
	\begin{aligned}
	(p \<'_\l q) \<'_{\l+\m} \b(r) & = (p \<_\l q + p ._\l q)\<'_{\l+\m}\b(r)
	\\&=(p \<_\l q + p ._\l q)\<_{\l+\m}\b(r) +(p \<_\l q + p ._\l q)._{\l+\m}\b(r)
	\\&=(p \<_\l q)\<_{\l+\m}\b(r) + (p ._\l q)\<_{\l+\m}\b(r) +(p \<_\l q )._{\l+\m}\b(r)+ (p ._\l q)._{\l+\m}\b(r)
	\\& =\a(p)\<_\l(q\<_\m r+q\>_\m r+q._\m r)+\a(p)._\l(q\<_\m r)+\a(p)._\l(q\>_\m r)+\a(p)._\l(q._\m r) 
	\\&= \a(p)\<_\l(q\<_\m r+q\>_\m r+q._\m r)+\a(p)._\l(q\<_\m r+q\>_\m r+ q._\m r)
	\\&=\a(p)\<'_\l(q\<_\m r+q\>_\m r+q._\m r)
	\\&= \a(p)\<'_\l(q\<'_\m r+q\>'_\m r).
	\end{aligned}
	\end{equation*} The above relation is satisfied by using  Eqs. (\ref{TD3}), (\ref{TD6}), (\ref{TD8}) and (\ref{TD9}).
	\begin{equation*}
	\begin{aligned}
	(p\>'_\l q) \<'_{\l+\m} \b(r) &= (p \>_\l q) \<_{\l+\m} \b(r) + (p \>_\l q) ._{\l+\m} \b(r)
	\\& = \a(p) \>_\l ( q \<_\m r) + \a(p) \>_\l (q ._\m r)
	\\&= \a(p) \>_\l (q \<_\m r + q ._\m r) \\&= \a(p) \>'_\l(q \<'_\m r).
	\end{aligned}\end{equation*}The above relation is satisfied using Eqs. (\ref{TD4}), and (\ref{TD7}). Now using (\ref{TD6}), we have
	\begin{equation*}
	\begin{aligned}
	\a(p) \>'_\l(q \>'_\m r) &= \a(p) \>_\l (q \>_\m r)
	\\&  = (p \<_\l q+ p \>_\l q + p ._\l q) \>_{\l+\m} \b(r)
	\\& = (p \<'_\l q + p \>'_\l q  ) \>'_{\l+\m} \b(r).
	\end{aligned}
	\end{equation*} It completes the proof.
\end{proof}
\begin{prop}
	Let $(A, \<_\l, \>_\l,._\l)$ be a tridendriform conformal algebra and $\a,\b\in cend(A)$. Define conformal bilinear maps $\<_{\l}^{(\a,\b)}, \>_{\l}^{(\a,\b)},._{\l}^{(\a,\b)}: A \otimes A \to A[\l]$ by $p \<_{\l}^{(\a,\b)} q = \a(p) \<_\l \b(q)$,  $p \>_{\l}^{(\a,\b)} q =\a(p) \>_\l \b(q)$ and $p ._{\l}^{(\a,\b)} q =\a(p) ._\l \b(q)$ for all $p,q\in A$. Then the Yau twist of $A$, $A_{(\a,\b)} := (A, \<_{\l}^{(\a,\b)}, \>_{\l}^{(\a,\b)}, ._{\l}^{(\a,\b)}, \a, \b)$ is  called a BiHom-tridendriform conformal algebra. Furthermore, suppose another BiHom-tridendriform conformal algebra $(A', \<'_{\l}, \>'_{\l},.'_{\l}, \a',\b')$ satisfying $f \circ \a = \a' \circ f$ and $f \circ \b = \b ' \circ f$. Then $f : A_{(\a,\b)}\to A'_{(\a',\b')}$ is a morphism of BiHom-tridendriform conformal algebras if $f:A\to A'$ is a tridendriform-conformal algebras morphism.
\end{prop}
\begin{proof}Proof is straightforward by using Proposition \ref{prop3.5}.
\end{proof}																		
\begin{rem}	In a broader sense, consider $(A, \<_\l,\>_\l,._\l, \a,\b)$ as a BiHom-tridendriform conformal algebra and $\tilde{\a}, \tilde{\b} \in cend(A)$ as two morphisms of BiHom-tridendriform conformal algebras in such a way that any two of the maps $\a, \b , \tilde{\a}, \tilde{\b} $ commute. For any  $p, q\in A$, define a new multiplication on $A$ by $p \<'_\l q= \tilde{\a}(p) \<_\l\tilde{\b}(q)$, $p \>'_\l q = \tilde{\a}(p) \>_\l \tilde{\b}(q)$  and $p .'_\l q= \tilde{\a}(p) ._\l \tilde{\b}(q)$.  Then it can be demonstrated that $(A, \<'_\l, \>'_\l , ._\l, \a\circ \tilde{\a}, {\b}\circ \tilde{\b})$ is a BiHom-tridendriform conformal algebra.
\end{rem}

\begin{prop}\label{proptri2}
	Consider a BiHom-tridendriform conformal algebra $(A, \<_\l,\>_\l,._\l, \a,\b)$. Then $(A, *_\l, \a,\b)$ is a BiHom-associative conformal algebra by the multiplication $*_\l : A \otimes A \to A[\l]$ defined by $p *_\l q  = p \<_\l q + p \>_\l q+ p._\l q$, for all $p,q\in A$.
\end{prop}
\begin{proof}
	By using Definition \ref{deftridendriform}, we have $\a(p *_\l q) = \a(p) *_\l \a(q)$ and $\b(p *_\l q) = \b(p) *_\l \b(q),$ for all $p, q\in A,$ . Now we evaluate, for $p, q, r\in A$:
	\begin{equation*}
	\begin{aligned}
	\a(p) *_\l (q *_\m r) &= \a(p) *_\l (q \<_\m r) + \a(p) *_\l (q \>_\m r) + \a(p) *_\l(q ._\m r)
	\\&=\a(p) \<_\l (q\<_\m r) + \a(p) \>_\l (q \<_\m r) + \a(p) ._\l (q \<_\m r)
	\\&	+\a(p) \<_\l (q \>_\m r) + \a(p) \>_\l (q \>_\m r) + \a(p) ._\l (q \>_\m r)
	\\&	+\a(p)\<_\l(q ._\m r) + \a(p) \>_\l (q ._\m r) + \a(p)._\l (q ._\m r)
	\\&	= \a(p) \<_\l (q \<_\m r + q \>_\m r +q._\m r)
	\\&+ \a(p) \>_\l (q \<_\m r) +\a(p) \>_\l (q \>_\m r) + \a(p) ._\l (q \<_\m r) \\&+ 
	\a(p) ._\l (q \>_\m r) 	+ \a(p) \>_\l (q ._\m r) + \a(p)._\l (q ._\m r)
	\\&	\overset{(\ref{TD3}),(\ref{TD4}),(\ref{TD5})}{=} (p \<_\l q) \<_{\l+\m} \b(r) + (p \>_\l q) \<_{\l+\m}\b(r) + (p \<_\l q) \>_{\l+\m} \b(r)
	\\&	+(p \>_\l q) \>_{\l+\m} \b(r) + (p ._\l q) \>_{\l+\m} \b(r) + \a(p)._\l (q \<_\m r)
	\\&	+\a(p) ._\l (q \>_\m r) + \a(p) \>_\l (q ._\m r) + \a(p) ._\l (q ._\m r)
	\\&\overset{(\ref{TD6}),(\ref{TD7}),(\ref{TD8}),(\ref{TD9})}{=}(p \<_\l q) \<_{\l+\m} \b(r) + (p \>_\l q) \<_{\l+\m} \b(r) + (p \<_\l q) \>_{\l+\m} \b(r)
	\\&	+(p \>_\l q) \>_{\l+\m}\b(r) + (p ._\l q)\>_{\l+\m}\b(r) + (p ._\l q) \<_{\l+\m}\b(r)
	\\&	+(p \<_\l q) ._{\l+\m}\b(r) + (p \>_\l q)._{\l+\m} \b (r) + (p._\l q)._{\l+\m}\b (r)
	\\&	= (p \<_\l q + p \>_\l q + p ._\l q) \<_{\l+\m}\b(r)+ (p \<_\l q + p \>_\l q + p._{\l} q)\>_{\l+\m}\b(r)\\&+ (p \<_{\l}q + p \>_\l q + p ._{\l} q) ._{\l+\m}\b(r)\\&= (p*_{\l} q) \<_{\l+\m}\b(r) + (p *_{\l} q)\>_{\l+\m}\b(r) + (p *_\l q) ._{\l+\m}\b (r) \\&= (p *_{\l} q) *_{\l+\m}\b (r).
	\end{aligned}
	\end{equation*} It completes the proof.
\end{proof}
\begin{cor}\label{corbihomdend}	Consider a BiHom-dendriform conformal algebra $(A, \<_\l,\>_\l, \a,\b)$. Then $(A, *_\l, \a,\b)$ is a BiHom-associative conformal algebra by the multiplication $*_\l : A \otimes A \to A[\l]$ defined by $p *_\l q  = p \<_\l q + p \>_\l q$, for all $p,q\in A$.\end{cor}
Now to show the generalizations of BiHom-NS-conformal algebra in term of (tri)dendriform conformal algebras, We first introduce the BiHom-analogue of classical NS-conformal algebra as follows,
\begin{defn}
	A  $6$-tuple $(A, \<_\l, \>_\l, \vee_\l, \a, \b)$ consisting of a $\mathbb{C}[\p]$-module $A$, three bilinear multiplication maps $\<_\l, \>_\l, \vee_\l : A \otimes A \to A[\l]$ and two commuting linear maps $\a, \b : A \to A$ is said to be a BiHom-NS-conformal algebra if along with the conformal sesqui-linearity, following conditions hold:\begin{eqnarray}
	&\a (p \<_\l q) = \a (p) \<_\l \a (q), \a (p \>_\l q) = \a (p) \>_\l \a (q), \a (p \vee_\l q) = \a (p) \vee_\l \a (q), \label{BN2.2}\\
	&\b(p \<_\l q) = \b(p) \<_\l \b(q), \b (p \>_\l q)=\b (p) \>_\l \b(q), \b (p \vee_\l q) = \b (p) \vee_\l \b (q), \label{BN2.3}\\ 
	&(p \<_\l q) \<_{\l+\m}\b (r) = \a (p) \<_\l (q*_{\m} r), \label{BN2.4}\\
	&(p \>_\l q) \<_{\l+\m} \b (r) = \a (p) \>_\l (q \<_\m r), \label{BN2.5}\\
	&(p *_\l q) \>_{\l+\m} \b (r) = \a (p) \>_\l (q \>_\m r), \label{BN2.6}\\
	&(p \vee_\l q) \<_{\l+\m}\b(r) + (p *_\l q)\vee _{\l+\m}\b(r) = \a(p) \>_\l (q \vee_\m r) + \a(p) \vee_\l (q *_\m r), \label{BN2.7}\\
	&p *_\l q=p \<_\l q + p\>_\l q + p \vee_\l q. \label{BN2.8}
	\end{eqnarray}for all $p,q,r \in A$.
\end{defn}A morphism of BiHom-NS- conformal algebras $f : (A, \<_\l, \>_\l,\vee_\l,\a,\b)\to (A',\<'_\l, \>'_\l,\vee'_\l,\a',\b')$  is a linear map $f : A \to A'$, that satisfies $f\circ \a= \a'\circ f$, $f\circ \b = \b'\circ f$, $f(p \<_\l q) = f(p) \<'_\l f(q), f(p \>_\l q) = f(p) \>'_\l f(q)$ and $f(p \vee_{\l}q) = f(p) \vee'_\l f(q),$ for all $p, q \in A.$ As a result, we have $f (p*_{\l} q) = f (p) *'_\l f (q)$,  for all $p,q\in A$. Now we characterize BiHom-NS-conformal algebras in term of conformal bimodule as follows.
\begin{prop}\label{prop3.2}
	Consider a $6$-tuple $(A, \>_\l, \<_\l,\vee_\l, \a,\b)$ equipping with  a $\mathbb{C}[\p]$-module $A$, three multiplication maps  $\<_\l, \>_\l, \vee_\l: A\otimes A \to A[\l]$ and two linear commuting maps $\a,\b : A \to A$ such that $\a$ and $\b$ are multiplicative with respect to $\>_\l, \<_\l$ and $\vee_\l$. Defining a new product on $A$ (for all $p, q\in A$) by $p *_\l q = p \<_\l q + p \>_\l q + p\vee_\l q,$ provides $(A,  \>_\l, \<_\l, \vee_\l, \a, \b)$ is a BiHom-NS-conformal algebra if and only if $A_{bhas}=(A, *_\l, \a, \b)$ is a BiHom-associative conformal algebra, and $(A, \<_\l,\>_\l,\a,\b)$ is a conformal $(A, *_\l, \a, \b)$- bimodule.
\end{prop}
\begin{proof}Assume that $(A, \<_\l, \>_\l, \vee_\l, \a, \b)$ is a BiHom-NS-conformal algebra. It is easy to see that $\a$ and $\b$ are multiplicative with respect to $*_\l$, so we only check conformal  BiHom-associative identity. We compute, \begin{equation*}
	\begin{aligned}
	\a(p) *_\l (q *_\m r)&
	= \a(p) \>_\l (q\<_\m r + q \>_\m r + q \vee_\m r) + \a(p) \<_\l(q *_\m r) + \a (p)_\l \vee (q *_\m r)\\
	&= \a (p) \>_\l (q \<_\m r) + \a (p) \>_\l (q \>_\m r) + \a(p) \>_\l (q\vee_\m r)
	+\a (p) \<_\l (q *_\m r) + \a (p) \vee_\l (q *_\m r)\\
	&\overset{(\ref{BN2.4}),(\ref{BN2.5}),(\ref{BN2.6})}{=}(p *_\l q) \>_{\l+\m} \b (r) + (p \>_\l q) \<_{\l+\m} \b (r) + [\a (p) \>_\l (q \vee_\m r) + \a (p) \vee_\l (q *_\m r)] + (p \<_\l q) \<_{\l+\m} \b (r)\\
	&\overset{(\ref{BN2.7})}{=}(p *_\l q) \>_{\l+\m}\b(r) + (p\>_\l q) \<_{\l+\m} \b (r) + ((p \vee_\l q) \<_{\l+\m} \b(r) + (p *_\l q)\vee_{\l+\m} \b (r)) + (p \<_\l q) \<_{\l+\m}\b (r)\\
	&= (p *_\l q) \>_{\l+\m}\b(r) + (p \<_{\l} q + p \<_\l q + p\vee_\l q) \<_{\l+\m} \b (r)
	+ (p *_\l q) \vee_{\l+\m} \b (r)\\
	&\overset{(\ref{BN2.8})}{=}(p *_\l q) \<_{\l+\m}\b (r) + (p *_\l q) \>_{\l+\m} \b (r) + (p *_\l q) \vee_{\l+\m} \b(r)\\
	&= (p *_\l q)*_{\l+\m} \b (r),\end{aligned}
	\end{equation*}	for all $p, q, r \in A.$ It proves that $(A, *_\l, \a, \b )$ is BiHom-associative conformal algebra. The fact that $(A, \>_\l, \<_\l, \a,\b)$ is a conformal $A_{bhas}$-bimodule follows immediately from Eqs. (\ref{BN2.4})-(\ref{BN2.6}).\newline Conversely, Eqs. (\ref{BN2.4})-(\ref{BN2.6}) follow from the fact that $(A, \>_\l, \<_\l, \a,\b)$ is a conformal $A_{bhas}$-bimodule, while Eq. (\ref{BN2.7}) is obtained by subtracting Eqs. (\ref{BN2.4})-(\ref{BN2.6}) from the conformal  BiHom-associative condition $\a(p) *_\l (q *_\m r) = (p*_\l q) *_{\l+\m}\b (r)$ written explicitly for $*_\l= \<_\l+ \>_\l+ \vee_\l$.
\end{proof}
Now we show that BiHom-NS-conformal algebras can be obtained by Yau-twisting of classical NS-conformal algebras as follows.
\begin{prop}
	Consider an NS-conformal algebra $(A,\<_\l, \>_\l, \vee_\l)$ along with two commuting NS-algebra morphisms $\a, \b : A \to A$. Define new multiplications $\<_{\l} ^{(\a, \b)}, \>_{\l} ^{(\a, \b)} ,\vee_{\l} ^{(\a,\b)}
	: A \otimes  A\to  A[\l]$ by $p \<_{\l} ^{(\a,\b)} q = \a(p)\<_\l\b(q), p \>_{\l} ^{(\a,\b)} q = \a(p) \>_\l \b(q), p\vee_{\l}^{(\a,\b)} q= \a (p) \vee_\l \b (q),$ for all $p, q \in A.$ Then $A_{(\a,\b)}:= (A, \>_{\l}^{(\a,\b)}, \<_{\l} ^{(\a,\b)}, \vee_{\l} ^{(\a,\b)}, \a, \b)$ is a BiHom-NS-conformal algebra, also called as Yau-twist of $A$. In addition, assume that $(A', \<'_\l, \>'_\l, \vee'_\l)$ is another NS-conformal algebra with two commuting NS-conformal algebra morphisms $\a', \b' : A'\to A'$ and a morphism of $NS$-conformal algebras $f : A \to A'$ that satisfies $\a'\circ f= f \circ \a$ and $ \b'\circ f=f \circ\b.$ Then $f : A_{(\a,\b)} \to A'_{(\a',\b')}$ is a morphism of BiHom-NS-conformal algebras.
\end{prop}\begin{proof}One can easily see that, if we define $*_{\l}^{(\a,\b)}:= \<_{\l} ^{(\a,\b)}+ \>_{\l} ^{(\a,\b)}+ \vee_{\l} ^{(\a,\b)}$, then $p *_\l^{(\a,\b)} q = \a (p) *_\l \b (q), \forall p, q \in A.$ The conditions in Eq. (\ref{BN2.2}) and Eq. (\ref{BN2.3}) are easy to prove and left to the reader. We verify other conditions as follows:
	\begin{equation*}\begin{aligned}
	(p \<_{\l}^{(\a,\b)} q) \<_{\l+\m}^{(\a,\b)}\b (r)&=\a(\a(p) \<_{\l} \b(q)) \<_{\l+\m}\b^2 (r)= (\a^2(p) \<_{\l} \a\b(q)) \<_{{\l+\m}}\b^2 (r)\\&= \a^2(p) \<_{\l} (\a\b(q) *_{\m}\b^2 (r))= \a^2(p) \<_{\l} (\b\a(q) *_{\m}\b^2 (r))\\&=\a^2(p) \<_{\l}\b (\a(q) *_{\m}\b (r))= \a (p) \<_\l^{(\a,\b)} (q*_{\m}^{(\a,\b)} r).
	\end{aligned}\end{equation*} The conditions (\ref{BN2.5}) and (\ref{BN2.6}) can be proved similarly. Now we prove, Eq. (\ref{BN2.7})
	\begin{equation*}
	\begin{aligned}
	(p \vee_\l q) \<_{\l+\m}^{(\a,\b)} \b (r) + (p *_{\l}^{(\a,\b)} q) \vee _{\l+\m}^{(\a,\b)}\b (r)&= \a(\a(p) \vee_\l \b(q)) \<_{\l+\m} \b^2 (r) + \a(\a(p) *_\l \b(q)) \vee _{\l+\m}\b^2 (r)\\& =(\a^2(p) \vee_\l \a\b(q)) \<_{\l+\m} \b^2 (r) + (\a^2(p) *_\l \a\b(q)) \vee _{\l+\m}\b^2 (r)\\&=\a^2(p) \>_\l \b(\a(q) \vee_{\m} \b (r)) + \a^2(p) \vee_\l \b(\a(q) *_{\m}\b (r))\\&=\a(p) \>_{\l}^{(\a,\b)} (q \vee_{\m}^{(\a,\b)} r) + \a(p) \vee_{\l}^{(\a,\b)} (q *_{\m}^{(\a,\b)}r).
	\end{aligned}
	\end{equation*}Thus, we proved that $A_{(\a,\b)}$ is a BiHom-NS-conformal algebra. The last statement is easy to prove and left to the reader.
\end{proof}
BiHom-NS-conformal algebras are actually a generalization of BiHom-tridendriform conformal algebras can be shown by the following Proposition.
\begin{prop}\label{prop3.4}Assume that $(A, \<_\l,\>_\l,._\l,\a,\b)$ is a BiHom-tridendriform conformal algebra. Then $(A, \<_\l,\>_\l,._\l,\a,\b)$ is a BiHom-NS-conformal algebra.
\end{prop}\begin{proof}Note that the first $6$ axioms defining a BiHom-tridendriform conformal algebra are identical to the first $6$ axioms defining a BiHom-NS-conformal algebra, so we only check the relation in Eq. (\ref{BN2.7}). By using definition \ref{deftridendriform}, we have:
	\begin{equation*}
	\begin{aligned}
	(p ._\l q) \<_{\l+\m }\b (r) + (p *_\l q) ._{\l+\m}\b(r)&= (p ._\l q) \<_{\l+\m} \b (r) + (p \>_\l q + p \<_\l q + p ._\l q) ._{\l+\m} \b (r)\\&= \a(p)._\l(q \<_\m r) + (p \>_\l q)._{\l+\m} \b (r) + (p \<_{\l} q) ._{\l+\m} \b (r) + (p ._{\l} q) ._{\l+\m} \b (r)\\
	&=\a(p) ._\l (q \<_\m r) + \a(p) \>_\l (q ._\m r) + \a(p)._\l (q \>_\m r) + \a(p) ._{\l} (q._\m r)\\&
	= \a(p) ._\l (q \<_\m r + q\>_\m r + q ._\m r) + \a(p)\>_\l(q ._\m r)\\&
	= \a (p)._\l(q*_\m r) + \a (p) \>_\l (q._\m r).
	\end{aligned}\end{equation*}
	It completes the proof.\end{proof} Our aim now is to give a more conceptual proof of Proposition \ref{prop3.4},  For this we first recall conformal bimodule algebras over BiHom-associative conformal algebras and use them to obtain a characterization of BiHom-tridendriform conformal algebras.

 Let $(A, \tau_\l^A, \a_A, \b_A)$ be a BiHom-associative conformal algebra,  $M$ be a $\mathbb{C}[\p]$-module, $\a_M,\b_M : M \to M$ are two commuting $\mathbb{C}$-linear maps and $l_\l : A\otimes  M \to M[\l],$ $p \otimes m \mapsto p ._\l m $ and $r_\l : M \otimes A \to M[\l]$, $m \otimes  p \mapsto m ._\l p$ are two $\mathbb{C}$-bilinear maps. On the direct sum $A \oplus M$ consider the algebra structure defined by $$(p, m)._\l(p',m') = (p._\l p', p._\l m' + m._\l p'),\quad  \forall p, p' \in A, m, m'\in M,$$ denoted by $A \oplus_{0} M$. Then $A \oplus_{0} M$ is a BiHom-associative conformal algebra with structure maps $\a, \b : A \oplus_{0} M\to A \oplus_{0} M$, $\a(p, m) = (\a_A(p), \a_M (m))$ and $\b(p, m)= (\b_A(p), \b_M (m)),$ if and only if $(M, l_\l, r_\l, \a_M , \b_M )$ is a conformal $A$-bimodule. Assume again that $(A, \tau^\l_A, \a_A, \b_A)$ is a BiHom-associative conformal algebra, where $M$, $\a_M$, $\b_M$, $l_\l$, $r_\l$ are as defined above, but now assume that we also have another map $\circledcirc_\l : M \otimes M \to M[\l]$.
\begin{defn} We say that $(M, \circledcirc_\l, l_\l, r_\l, \a_M , \b_M )$ is a conformal $A$-bimodule algebra if $(A\oplus M, *_{\circledcirc_\l}, \a, \b)$ is a BiHom-associative conformal algebra, where the direct sum structure is defined by $$(p, m)*_{\circledcirc_\l}(p', m') = (p_\l p' , p._\l m'+ m._\l p' + m \circledcirc_\l m')$$ and $\a(p, m) = (\a_A(p), \a_M (m)), \b(p, m) = (\b_A(p), \b_M (m)),$ for all $p, p'\in A$ and $m, m'\in M$. \end{defn}Now, we show that, this concept is indeed a BiHom-analogue of the corresponding concept introduced for associative conformal algebras in \cite{2}.
\begin{prop} With the above notations, $(M, \circledcirc_\l, l_\l, r_\l, \a_M , \b_M)$ is a conformal $A$-bimodule algebra if and only if \newline $(M, l_\l, r_\l, \a_M, \b_M)$ is a conformal $A$-bimodule, $(M, \circledcirc_\l, \a_M , \b_M)$ is a BiHom-associative conformal algebra and satisfies the following identities, for all $p\in  A$ and $m, m'\in M$ 
	\begin{equation}\label{eqbimodule} \begin{aligned}\a_A(p)._\l (m \circledcirc_\m m')= (p ._\l m) \circledcirc_{\l+\m} \b_M (m'),\\ \a_M(m) \circledcirc_{\l} (m'._\m p) = (m \circledcirc_\l m') ._{\l+\m}\b_A(p), \\ \a_M(m) \circledcirc_\l (p ._\m m') = (m ._\l p) \circledcirc_{\l+\m} \b_M (m'). \end{aligned} \end{equation} \end{prop}
\begin{proof}  We only sketch the proof and leave the details to the reader. The conformal BiHom-associativity condition for $(A \oplus M, *_{\circledcirc_\l}, \a, \b)$ written for elements $(p, m),(p',m'),(p'',m'') \in A \oplus M$ turns out to be equivalent to\begin{equation*}
		\begin{aligned}
		\a_A(p) ._\l (p.'_\m m'') + \a_A(p) ._\l (m'._\m p'') + \a_A(p)._\l(m'
		\circledcirc_\m m'') \\+ \a_M (m) ._\l (p.'_\m p'')
		+\a_M(m) \circledcirc_\l (p'._\m m'') + \a_M(m) \circledcirc_\l(m'
		._\m p'') + \a_M(m) \circledcirc_{\l} (m'\circledcirc_\m m'')\\
		= (p._\l p'
		)._{\l+\m}\b_M (m'') + (p._\l m'
		) ._{\l+\m} \b_A(p'') + (m _\l p') ·_{\l+\m} \b_A(p'') \\+ (m \circledcirc_\l m'
		) ._{\l+\m} \b_A(p'')
		+(p ._\l m'
		) \circledcirc_{\l+\m} \b_M(m'') + (m._\l p' ) \circledcirc_{\l+\m} \b_M(m'') + (m\circledcirc_\l m'
		) \circledcirc_{\l+\m} \b_M(m'')
		\end{aligned}
		\end{equation*}By taking $p = p'=p''=0$ we obtain $\a_M(m) \circledcirc_\l (m' \circledcirc_{\m} m'') = (m\circledcirc_\l m') \circledcirc_{\l+\m} \b_M(m'')$. By taking
		$p'' = 0,$ $m = m' = 0$, then $p = 0, m' = m'' = 0$ and then $p'= 0,$ $m = m'' = 0,$ we
		respectively obtain $\a_A(p) ._\l (p'._\m m'') = (p._\l p') ._{\l+\m}\b_M (m''), \a_M(m) ._\l (p '._{\m}p'') = (m ._\l p' )._{\l+\m}\b_A(p'')$ and
		$\a_A(p)._\l(m'._{\m} p'') = (p ._{\l}m' )._{\l+\m} \b_A(p'')$. Thus, $(M, l_{\l}, r_{\l}, \a_M , \b_M)$ is a conformal $A$-bimodule. Similarly one can obtain the relations given in Eq. (\ref{eqbimodule}).\end{proof}
	Similar to the the characterization of BiHom-NS-conformal algebras in Proposition \ref{prop3.2}, we can now characterize  BiHom-tridendriform conformal algebras as follows \begin{prop} Consider a $6$-tuple $(A, \<_\l, \>_\l, ._\l, \a, \b)$ consisting of a $\mathbb{C}[\p]$-module $A$, multiplication maps $\<_\l,\>_\l,._\l : A \otimes A \to A[\l]$ and linear maps $\a,\b : A\to A$ such that $\a$ and $\b$ are multiplicative with respect to $\<_\l, \>_\l$ and $._\l$.  Then $(A, \<_\l,\>_\l,._\l, \a,\b)$ is a BiHom-tridendriform conformal algebra with a new multiplication $p *_\l q = p \<_\l q + p \>_\l q + p._\l q,$ for all $p, q \in  A$, if and only if $(A, *_\l,\a,\b)$ is a BiHom-associative conformal algebra and  $(A, \<_\l,\>_\l,._\l, \a,\b)$ is  a conformal $(A, *_\l, \a,\b)$-bimodule algebra.
	\end{prop}
\begin{proof}Assume that $(A, \<_\l, \>_\l, ._\l, \a, \b)$ is BiHom-tridendriform conformal algebra. The tuple $(A, *_\l, \a,\b)$  defines a BiHom-associative conformal algebra by Proposition \ref{proptri2}, so we only have to prove that $(A, \<_\l, \>_\l, ._\l, \a, \b)$ is a conformal $(A, *_\l, \a,\b)$-bimodule algebra.\\ The fact that $(A,._\l, \a, \b)$ is BiHom-associative conformal is equivalent to Eq. (\ref{TD9}), while each of conditions in Definition \ref{def2.2} and Eq. (\ref{eqbimodule}) is equivalent to one of the condition in Eqs. (\ref{TD3})-(\ref{TD8}). These observations indicate that the converse also holds true.
 \end{proof}
\section{ (Twisted-)Rota-Baxtrer operator on BiHom-associative conformal algebra}
\begin{defn} 
	Let $A$ be a $\mathbb{C}[\p]$-module with bilinear multiplication $\tau_\l : A \otimes A \to A[\l]$ and let $\a,\b$ are two commuting linear maps on $A$, correspondingly a homomorphism $R: A\to A$ of $\mathbb{C}[\p]$-module is called Rota-Baxter operator of weight $\theta$, if for all $p, q\in A $ and $ \theta\in\mathbb{C}$, the following relation holds 
	\begin{equation}\label{eq1}
	R(p)_\l R(q) =R(R(p)_\l q+ p_\l R(q)+ \theta(p_\l q)). 
	\end{equation}
\end{defn}
\begin{prop} Consider a $\mathbb{C}[\p]$-module $A$ , a conformal multiplication $\tau_\l : A \otimes A \to A[\l]$on $A$, 
	 a Rota-Baxter operator $R : A\to A$ of weight $\theta$ for $(A,\tau_\l)$ and two linear maps $\a,\b : A \to A$ such that $\a\circ R= R\circ \a $ and $\b\circ R=R\circ \b$. Defining a new product on $A$ by $p*_\l q = \a(p)_\l\b(q),$ for all $p,q\in A.$ Then $R$ is also a Rota-Baxter operator of weight $\theta$ for $(A, *_\l)$. Specifically, if $(A, \tau_\l)$ is an associative conformal algebra and $\a, \b$ are two commuting  endomorphisms, then $R$ is a Rota-Baxter operator of weight $\theta$ on the BiHom-associative conformal algebra $(A,\tau_\l\circ (\a\otimes \b), \a, \b)$ denoted by $A_\l^{(\a,\b)}$.
\end{prop}
\begin{proof}We compute:
	\begin{equation*}\begin{aligned}
	R(p)*_\l R(q)=\a(R(p))_\l\b(R(q)) &= R(\a(p))_\l R(\b(q))
	\\&= R(R(\a(p))_\l \b(q) + \a(p)_\l R(\b(q))+\theta(\a(p)_\l\b(q)))
	\\&= R(\a(R(p))_\l \b(q) + \a(p)_\l \b(R(q))+\theta(\a(p)_\l\b(q)))
	\\&= R(R(p) *_\l q + p *_\l R(q)+ \theta(p*_\l q)).
	\end{aligned}
	\end{equation*}It completes the proof.
\end{proof}
\begin{prop}\label{proptri} Consider a Rota-Baxter operator $R : A \to A$ of weight $\theta$ on the  BiHom-associative conformal algebra $(A, \tau_\l,\a,\b)$ that commute with $\a$ and $\b$. Then a BiHom-tridendriform conformal algebra $(A, \<_\l,\>_\l,._\l,\a,\b)$ can be obtained by $p \<_\l q = p_\l R(q)$, $p \>_\l q= R(p)_\l q$ and $p._\l q = \theta(p_\l q)$ for all $p, q \in A$.\end{prop}
\begin{proof}
	To show that $(A,\<_\l,\>_\l,._\l ,\a,\b)$ is a BiHom-tridendriform conformal algebra, we staisfy the following identities of the Definition \ref{deftridendriform}.
	\begin{equation*}
	\begin{aligned}
	&(p \<_\l q) \<_{\l+\m} \b(r) - \a(p) \<_\l (q \<_\m r) -\a(p) \<_\l (q \>_\m r) -\a(p) \<_\l (q ._\m r)\\&
	= (p_\l R(q)) \<_{\l+\m} \b(r) -\a(p) \<_\l (q_\m R(r)) -\a(p) \<_\l(R(q)_\m r) -\a(p) \<_\l (\theta q_\m r)\\&
	= (p_\l R(q))_{\l+\m}R(\b(r))-\a(p)_\l R(q_\m R(r))-\a(p)_\l R(R(q)_\m r) -\a(p)_\l R(\theta q_\m r)
	\\&= (p_\l R(q))_{\l+\m}\b(R(r))-\a(p)_\l R(q_\m R(r))-\a(p)_\l R(R(q)_\m r) -\a(p)_\l R(\theta q_\m r)
	\\&= \a(p)_\l(R(q)_\m R(r))-\a(p)_\l R(q_\m R(r) + R(q)_\m r + \theta q_\m r) \\&= 0.
	\end{aligned}
	\end{equation*}
	\begin{equation*}
	\begin{aligned}
(p\>_\l q) \<_{\l+\m}\b(r) -\a(p) \>_{\l} (q \<_{\m} r)
	&= (R(p)_\l q) \<_{\l+\m}\b(r) -\a(p)\>_{\l} (q_\m R(r))
	\\& = (R(p)_\l q)_{\l+\m}R(\b(r)) - R(\a(p))_\l(q_\m R(r))
	\\&= (R(p)_\l q)_{\l+\m}\b(R(r)) - \a(R(p))_{\l}(q_\m R(r))
	\\&= \a(R(p))_{\l}(q_\m R(r)) - \a(R(p))_\l(q_\m R(r))\\&= 0.
	\end{aligned}\end{equation*}	
	\begin{equation*}
	\begin{aligned}
	&\a(p) \>_\l (q \>_\m r) - (p \<_\l q) \>_{\l+\m}\b(r) - (p \>_\l q) \>_{\l+\m}\b(r) - (p ._\l q) \>_{\l+\m} \b(r)
	\\&= \a(p)\>_\l (R(q)_\m r) - (p_\l R(q)) \>_{\l+\m} \b(r) - (R(p)_\l q) \>_{\l+\m}\b(r) -(\theta p_\l q) \>_{\l+\m} \b(r)
	\\&= R(\a(p))_\l(R(q)_\m r) - R(p_\l R(q))_{\l+\m}\b(r) - R(R(p)_\l q)_{\l+\m}\b(r) - R(\theta p_\l q)_{\l+\m}\b(r)
	\\&= \a(R(p))_\l(R(q)_\m r) - (R(p)_{\l}R(q))_{\l+\m}\b(r) \\&= (R(p)_\l R(q))_{\l+\m}\b(r) - (R(p)_\l R(q))_{\l+\m}\b(r) \\&= 0.
	\end{aligned}
	\end{equation*}
	Other identities are easy to prove and left to the reader. This completes the proof.	
\end{proof}
\begin{cor}\label{cor6.3} Consider a  Rota-Baxter operator $R : A \to A$ of weight $0$ on the  BiHom-associative conformal algebra $(A, \tau_\l,\a,\b)$ that commute with $\a$ and $\b$. Then a BiHom-dendriform conformal algebra $(A, \<_\l,\>_\l,\a,\b)$ can be obtained by $p \<_\l q = p_\l R(q)$ and $p \>_\l q= R(p)_\l q,$ for all $p, q \in A$.
\end{cor}As a consequence of Propositions \ref{proptri1} and \ref{proptri}, we obtain:
\begin{cor}Let $(A, \tau_\l,\a,\b)$ be a BiHom-associative conformal algebra and $R : A \to A$ a Rota-Baxter operator of weight $\theta$ such that $R \circ \a = \a\circ R$ and $R\circ \b = \b\circ R$. Define operations $\<'_\l$ and $\>'_\l$ on $A$ by $p \<'_\l q = p_\l R(q) + \theta(p_\l q)$ and $p \>'_\l q= R(p)_\l q,$ for all $p, q \in A.$ Then $(A, \<'_\l, \>'_\l, \a, \b)$ is a BiHom-dendriform conformal algebra.
\end{cor}
As a consequence of Propositions \ref{proptri2} and \ref{proptri} we obtain:
\begin{cor} Let $(A, \tau_\l,\a,\b)$ be a BiHom-associative conformal algebra and $R : A \to A$ a Rota-Baxter operator of weight $\theta$ on a BiHom-associative conformal algebra, such that $R \circ \a = \a\circ R$ and $R\circ \b = \b\circ R$. If we define on $A$, a new multiplication by $p *_\l q = p_\l R(q) + R(p)_\l q + \theta(p_\l q),$ for all $p, q\in A,$ then $(A, *_\l, \a, \b)$ is a BiHom-associative conformal algebra.
\end{cor}Now we introduce the following concept:
\begin{defn}\label{def6.6} Let $(D,\<_\l, \>_\l,\a,\b)$ be a BiHom-dendriform conformal algebra. A Rota-Baxter operator of weight $0$ on $D$ is a linear map $R : D\to D$ such that $R \circ \a = \a\circ R$ and $R\circ \b = \b\circ R$ and the following conditions are satisfied, for all $p, q \in D$: \begin{eqnarray}&R(p)\>_\l R(q) = R(p \>_\l R(q) + R(p) \>_\l q), \label{eq59}\\& 
	R(p) \<_\l R(q) = R(p \<_\l R(q) + R(p) \<_\l q).\label{eq60}
	\end{eqnarray}
\end{defn}
From the Corollary \ref{corbihomdend}, we know that $(D, *_\l,\a,\b)$ is a BiHom-associative conformal algebra. Adding the Eqs. (\ref{eq59}) and (\ref{eq60}) in the above definition, we obtain that $R$ is also a Rota-Baxter operator of weight $0$ for $(D,*_\l)$:\begin{equation*}R(p) *_\l R(q) = R(p *_\l R(q)+ R(p)*_\l q).
\end{equation*}

Now, we introduce the concept of twisted Rota-Baxter operator in the framework of BiHom-associative conformal algebra and show its relation with BiHom-NS-conformal algebras. Let $(A, \tau_\l^A, \a_A, \b_A)$ be a BiHom-associative conformal algebra and $(M, l_\l, r_\l, \a_M , \b_M )$ be a conformal $A$-bimodule. A Hochschild $2$-cocycle on $A$ with coefficients in $M$ is a skew symmetric bilinear map $H_\l : A \otimes A \to M[\l]$, that satisfies :
\begin{eqnarray}
&H_\l \circ (\a_{A} \otimes \a_{A}) = \a_{M }\circ H_\l, \quad H_\l \circ (\b_{A} \otimes \b_{A})= \b_{M} \circ H_\l, \label{5.1}\\&
{\a_{A}(p)}_{\l} H_{\m}(q, r)-  H_{\l+\m}(p_\l q, \b_A(r)) + H_\l(\a_A(p),q_\m r)- H_\l(p, q)_{\l+\m}\b_A (r) = 0\label{5.2}.\end{eqnarray}for all $p, q, r \in A$
\begin{defn}
	A linear map $R : M \to A$ is said to be an $H$-twisted Rota-Baxter operator, if for all $m, n\in M$, the following identities hold:
	\begin{eqnarray}R\circ \a_M=\a_A\circ R&,& R\circ \b_M=\b_A\circ R\label{5.3}
	\\ R (m)_\l R(n) &=& R(m _\l R(n) + R (m) _\l n + H_\l(R (m), R (n))).\label{5.4}
	\end{eqnarray}Where $H_\l$ is a Hochschild $2$-cocycle , $M$ is a conformal $A$-bimodule over an associative conformal algebra $A$. 
\end{defn} 
\begin{prop}
	Let $(A, \tau_\l^{A}, \a_A, \b_A) $ be a BiHom-associative conformal algebra, let $(M, l_\l, r_\l, \a_M, \b_M)$ be a conformal $A$-bimodule. Assuming  $H_\l : A \otimes A \to M[\l]$ be a Hochschild $2$-cocycle and  $R : M \to A $ be a $H$-twisted Rota-Baxter operator, we can define the following operations on $A$, for all $m, n \in M$: \begin{equation} \begin{aligned}
	m \>_\l n = R(m)_\l n, ~~~ m\<_\l n = m_\l R(n), ~~~~~ m \vee_{\l} n = H_\l(R(m), R(n)).\end{aligned}
	\end{equation} Then $(M, \<_\l, \>_\l, \vee_\l, \a_M , \b_M)$ is a BiHom-NS-conformal algebra.
\end{prop}\begin{proof}We need to check the relations in Definition of BiHom-NS conformal algebra, let us assume $p, q, r \in M$, we have:
	\begin{equation*} \begin{aligned} (p \<_\l q)\<_{\l+\m} \b_M(r) &= (p _\l R(q))_{\l+\m} R(\b_M (r))\\& \overset{(\ref{5.3})}{=}(p_\l R(q)) _{\l+\m} \b_A(R(r))\\&\overset{Def.(\ref{def2.2})}{=} \a_M (p) _\l (R(q)_\m R(r))\\&\overset{(\ref{5.4})}{=}\a_M(p) _\l R(R(q) _\m r + q _\m R(r) + H_\m(R(q), R(r)))\\ &= \a_M (p) _\l R(q\>_\m r + q \<_\m r + q\vee_\m r) \\&= \a_M(p)\<_\l(q *_\m r). \end{aligned} \end{equation*}
	\begin{equation*}
	\begin{aligned}(p \>_\l q) \<_{\l+\m} \b_M (r) &= (R(p) _\l q)_{\l+\m}R(\b_M (r)) \\&\overset{(\ref{5.3})}{=} (R(p) _\l q)_{\l+\m}\b_A(R(r))
	\\&\overset{Def.(\ref{def2.2})}{=} \a_A(R(p))_\l(q _\m R(r))\\& \overset{(\ref{5.3})}{=}
	= R(\a_M (p)) _{\l} (q _\m R(r)) \\&= \a_M(p) \>_\l (q\<_\m r).
	\end{aligned}
	\end{equation*}
	\begin{equation*}
	\begin{aligned}
	(p *_\l q) \>_{\l+\m} \b_M(r) &= (R(p)_\l q + p_\l R(q) + H_\l(R(p), R(q)))\>_{\l+\m}\b_M(r)\\&= R(R(p)_\l q +p_\l R(q) + H_\l(R(p), R(q)))_{\l+\m}\b_M (r)\\&
\overset{(\ref{5.4})}{=} (R(p)_\l R(q))_{\l+\m}\b_M (r)\\&
	 = \a_A(R(p))_\l (R(q)_\m r)\\&
	\overset{(\ref{5.3})}{=}R(\a_M (p))_\l(R(q)_\m r)\\&= \a_M(p)\>_\l (q \>_\m r).\end{aligned}\end{equation*}
	Finally, we check Eq. (\ref{BN2.7}):
	\begin{equation*}
	\begin{aligned}
	& \a_M(p)\>_\l (q \vee_\m r) - (p*_\l q)\vee_{\l+\m} \b_M (r) + \a_M(p) \vee_\l (q *_\m r) - (p \vee_\l q) \<_{\l+\m} \b_M (r)                                  \\
	& =  R(\a_M (p))_{\l} H_\m( R(q),  R(r))- (p \>_\l q + p \<_\l q + p \vee_\l q) \vee_{\l+\m} \b_M (r)                                                            \\
	& + \a_M(p) \vee_\l (q\>_\m r + q \<_\m r + q \vee_\m r) - H_\l( R(p),  R(q)) \<_{\l+\m} \b_M (r)                                                               \\
	& = R(\a_M (p))_\l H_\m(R(q), R(r))- H_{\l+\m}(R(R(p)_\l q + p _\l  R(q) + H_\l(R(p), R(q))), R(\b_M(r)))                                                \\
	& +H_\l(R(\a_M (p)), R(R(q)_\m r + q _\m R(r) + H_\m(R(q), R(r)))) - H_\l(R(p), R(q))_{\l+\m} R(\b_M (r))                                               \\
	& \overset{(\ref{5.4})}{ = }R(\a_M (p))_\l H_\m(R(q), R(r))- H_{\l+\m}( R(p)_\l R(q), R(\b_M (r)))+H_\l( R(\a_M (p)),  R(q)_\m R(r)) - H_\l( R(p),  R(q))_{\l+\m} R(\b_M(r))   \\
	& \overset{(\ref{5.3})}{ = }\a_A(R(p))_\l H_\m(R(q), R(r)) - H_{\l+\m}( R(p)_\l R(q), \b_A( R(r)))
	+H_{\l}(\a_A(R(p)), R(q)_\m R(r)) - H_{\l}(R(p), R(q))_{\l+\m}\b_A(R(r)) \\
	& \overset{(\ref{5.2})}{ = } 0.
	\end{aligned}
	\end{equation*}This completes the proof. \end{proof} 
\begin{lem}With reference to Proposition \ref{prop3.2}, assume that $(A,\<_\l, \>_\l, \vee_\l, \a,\b)$ be a BiHom-NS-conformal algebra, $A_{bhas}$ be a BiHom-associative conformal algebra  and  $(A,\<_\l,\>_\l, \a,\b)$ be a conformal $A_{bhas}$-bimodule. Define a bilinear map $H_\l: A \otimes A \to  A[\l],$ by $ H_\l(p, q) = p\vee_\l q,$ for all $p, q\in A$. Then $H$ is a Hochschild $2$-cocycle on $A_{bhas}$ with values in $(A,\<_\l,\>_\l, \a,\b)$.\end{lem}
\begin{proof} The Eq.(\ref{5.1}) for $H_\l$ follows immediately by the multiplicativity of $\a$ and $\b$ with respect to the operation $\vee_\l$. We need to check the cocycle condition in Eq. (\ref{5.2}), as follows:
	\begin{equation*}\begin{aligned}
	&\a(p) \>_\l H_\m(q, r) - H_{\l+\m}(p *_\l q, \b(r)) + H_\l(\a(p), q *_\m r - H_\l(p, q)\<_{\l+\m} \b(r)
	\\&= \a(p) \>_\l (q \vee_\m r)-(p *_\l q) \vee_{\l+\m} \b(r) + \a(p)\vee_{\l}(q *_\m r) - (p \vee_\l q)\<_{\l+\m} \b(r)\\&\overset{(\ref{BN2.7})}{=} 0.
	\end{aligned}\end{equation*}
	 It completes the proof.\end{proof}

\section{ BiHom-quadri conformal algebra} In this section, we deal with the BiHom-verson of quadri conformal algebras. That is indeed defined  for the first time here. We can get the corresponding results by modifying structural maps (for the Hom-quadri conformal algebra, we just need to consider $\a= \b$, where the case of quadri conformal algebras can be obtained by considering $\a= \b= id$, given that $\a$ and $\b$ are structural maps). Moreover we discuss relation betweenn quadri conformal algebra and BiHom-(tri)dendriform conformal algebras.
\begin{defn}\label{BHQCA}
	A BiHom-quadri-conformal algebra is a $7$-tuple $(Q,\nwarrow_\l, \swarrow_\l,\nearrow_\l,\searrow_\l,\a,\b)$ consisting of a $\mathbb{C}[\p]$-module $Q$, conformal bilinear maps $\nwarrow_\l, \swarrow_\l,\nearrow_\l,\searrow_\l: Q\otimes Q \to Q[\l]$ and linear commuting maps $\a,\b: Q \to Q$ satisfying conformal sesqui-linearity and the following axioms
	\begin{align}
	\a(p\searrow_\l q)=\a(p)\searrow_\l \a(q),\quad\a(p\swarrow_\l q)=\a(p)\swarrow_\l \a(q),\quad \a(p\nwarrow_\l q)=\a(p)\nwarrow_\l \a(q),\quad \a(p\nearrow_\l q)=\a(p)\nearrow_\l \a(q)\quad\quad\label{eq43}
	&\\\b(p\swarrow_\l q)=\b(p)\swarrow_\l \b(q),\quad \b(p\nwarrow_\l q)=\b(p)\nwarrow_\l \b(q),\quad \b(p\searrow_\l q)=\b(p)\searrow_\l \b(q),\quad
	\b(p\nearrow_\l q)=\b(p)\nearrow_\l \b(q)\quad\quad\label{eq44}&\\
	 (p \nwarrow_\l q) \nwarrow_{\l+\m} \b(r) = \a(p) \nwarrow_\l (q *_{\m} r),\quad(p \nearrow_\l q) \nwarrow_{\l+\m} \b(r) = \a(p) \nearrow_\l (q \<_{\m} r),\quad (p \wedge_\l q) \nearrow_{\l+\m} \b(r) = \a(p) \nearrow_\l (q \>_{\m} r)\quad\quad\label{eq45}&\\
	(p \swarrow_\l q)\nwarrow_{\l+\m} \b(r) = \a(p) \swarrow_\l (q \wedge_\m r),
\quad(p \vee_\l q) \nearrow_{\l+\m} \b(r) = \a(p) \searrow_\l (q \nearrow_\m r),
\quad(p \searrow_\l q) \nwarrow_{\l+\m} \b(r) = \a(p) \searrow_{\l} (q\nwarrow_\m r)\quad\quad\label{eq46}& \\
	(p \<_\l q)\swarrow_{\l+\m} \b(r) = \a(p) \swarrow_\l (q \vee_{\m} r),
\quad(p *_\l q) \searrow_{\l+\m} \b(r) = \a(p) \searrow_\l (q \searrow_\m r),
\quad(p \>_\l q) \swarrow_{\l+\m} \b(r) = \a(p) \searrow_\l(q \swarrow_{\m} r)\quad\quad\label{eq47}&
	\end{align}for all $\l,\m\in \mathbb{C}$ and $p, q, r\in Q$. Moreover, we have the following operations:
	\begin{eqnarray} & p\vee_\l q := p \swarrow_\l q + p \searrow_\l q, \label{eq50}\\
	&p\wedge_\l q := p \nwarrow_\l q + p \nearrow_\l q,\label{eq51}\\
	& p\<_\l q:= p \searrow_\l q + p \nearrow_\l q \label{eq48}\\
	&p \>_\l q := p \nwarrow_\l q + p \swarrow_\l q, \label{eq49}\\ 
	&p *_\l q : = p \swarrow_\l q + p \nwarrow_\l q+ p \searrow_\l q+ p \nearrow_\l q= p\vee_\l q+p\wedge_\l q= p \>_\l q+p\<_\l q.\label{eq52}
	\end{eqnarray} \end{defn}
A morphism of BiHom-quadri-conformal algebras $f : (Q,\nwarrow_\l, \swarrow_\l,\nearrow_\l,\searrow_\l,\a,\b) \to(Q', \nwarrow'_\l, \swarrow'_\l, \nearrow'_\l, \searrow'_\l, \a', \b')$ is a $\mathbb{C}$-linear map $f : Q \to Q'$ that satisfies $f(p \swarrow_\l q)= f(p) \swarrow'_\l f(q), f(p \searrow _\l q)= f(p) \searrow'_\l f(q), f(p \nwarrow_\l q) = f(p) \nwarrow'_\l f(q)$ and $f(p\nearrow_\l q) = f(p)\nwarrow'_\l f(q)$, for all $p, q\in Q,$ as well as $f\circ \a = \a'\circ  f $ and  $f \circ \b = \b'\circ f.$
\begin{prop}Let $(Q,\nwarrow_\l, \swarrow_\l, \nearrow_\l, \searrow_\l)$ be a quadri-conformal algebra and $\a,\b: Q\to Q$ are two commuting quadri-conformal algebra endomorphisms. Define $\mathbb{C}$-bilinear maps $\nwarrow_{\l}^{(\a,\b)}, \swarrow_{\l}^{(\a,\b)}, \nearrow_{\l}^{(\a,\b)}, \searrow_{\l}^{(\a,\b)}: Q\otimes Q \to Q[\l]$ by $p \nwarrow_{(\a,\b)}^{\l} q= \a(p)\nwarrow_\l \b(q), p\searrow_{\l}^{(\a,\b)} q = \a(p) \searrow_\l \b(q), p \swarrow_{\l}^{(\a,\b)} q = \a(p) \swarrow_\l \b(q), p \nearrow _{\l}^{(\a,\b)} q= \a(p)\nearrow_\l \b(q),$ for all $p, q \in Q.$ Then $Q_{(\a,\b)}:=(Q,\nwarrow _{\l}^{(\a,\b)} , \swarrow_{\l}^{(\a,\b)}, \nearrow _{\l}^{(\a,\b)}, \searrow _{\l}^{(\a,\b)}, \a, \b)$ is a BiHom-quadri conformal algebra, called the Yau twist of $Q$. Moreover, assume that $(Q',\nwarrow_\l', \swarrow_\l',\nearrow_\l',\searrow_\l')$ is another quadri-conformal algebra and $\a',\b':Q\to Q$ are two commuting quadri-conformal algebra endomorphisms and $f: Q \to Q'$ is a morphism of quadri-conformal algebras satisfying $f \circ \a = \a'\circ f$ and $f\circ \b=\b'\circ f$. Then $f : Q_{\a,\b}\to Q'_{\a',\b'} $ is a morphism of BiHom-quadri conformal algebras.\end{prop}
\begin{proof}We only prove Eq. (\ref{eq45}) and other conditions can be prove similarly. We define\begin{eqnarray*}
	p \<_{\l}^{(\a,\b)} q  := p \nwarrow_{\l}^{(\a,\b)} q + p \swarrow_{\l}^{(\a,\b)} q,\\ p \>_{\l}^{(\a,\b)} q:= p \nearrow_{\l}^{(\a,\b)} q + p\searrow_{\l}^{(\a,\b)} q,\\ p \vee _{(\a,\b)\l} q := p \searrow_{\l}^{(\a,\b)} q + p \swarrow_{\l}^{(\a,\b)} q,\\ p \wedge_{\l}^{(\a,\b)} q := p \nearrow_{\l}^{(\a,\b)} q + p \nwarrow_{\l}^{(\a,\b)} q,\\
	p*_{\l}^{(\a,\b)} q := p \nwarrow_{\l}^{(\a,\b)} q + p\swarrow_{\l}^{(\a,\b)} q + p \searrow_{\l}^{(\a,\b)} q + p \nearrow_{\l}^{(\a,\b)} q.
	\end{eqnarray*} for all $p, q \in Q$.  It is easy to get \begin{eqnarray*}
	p \>_{\l}^{(\a,\b)} q = \a(p) \>_{\l} \b(q), p \<_{\l}^{(\a,\b)} q = \a(p)\<_\l \b(q), \\p \vee_{\l}^{(\a,\b)} q = \a(p) \vee_\l \b(q), p \wedge_{\l}^{(\a,\b)} q = \a(p) \wedge_\l \b(q),\\ p *_{\l}^{(\a,\b)}q =\a(p) *_\l \b(q)
	\end{eqnarray*} for all $p, q\in  Q$. By using the fact that $\a$ and $\b$ are two commuting endomorphisms, one can compute, for all $p,q, r \in  Q:$ \begin{equation*}
	\begin{aligned}
	(p \nwarrow_{\l}^{(\a,\b)} q)\nwarrow_{\m+\l}^{(\a,\b)} \b(r)&= \a(\a(p) \nwarrow_{\l} \b(q))\nwarrow_{\m+\l} \b^2(r)= (\a^2 (p) \nwarrow_\l \a\b(q)) \nwarrow_{\m+\l} \b^2(r)\\&=\a^2 (p) \nwarrow_\l (\a\b(q) *_{\m} \b^2(r))=\a (p) \nwarrow_{\l}^{(\a,\b)} (q *_{\m}^{(\a,\b)} r)
	\end{aligned}
	\end{equation*}\begin{equation*}
	\begin{aligned}
	(p \nearrow_{\l}^{(\a,\b)} q)\nwarrow_{\m+\l}^{(\a,\b)}\b(r)&=(\a^2(p) \nearrow_\l \a\b(q)) \nwarrow_{\m+\l} \b^2(r)\\&=\a^2(p)\nearrow_\l (\a\b(q)\<_\m \b^2(r))\\&=\a(p) \nearrow_{\l}^{(\a,\b)} (q \<_{\m}^{(\a,\b)} r).
	\end{aligned}
	\end{equation*}Thus, Eq. (\ref{eq45}) verified by considering the elements $\a^2(p), \a\b(q)$ and $ \b^2(r)$ for quadri-conformal algebra case. \end{proof}

\begin{rem}More generally, let $(Q, \nwarrow_\l, \swarrow_\l, \nearrow_\l, \searrow_\l, \a, \b)$ be a BiHom-quadri conformal algebra and $\tilde{\a} ,\tilde{\b}: Q \to Q$ two morphisms of BiHom-quadri conformal algebras such that any two of the maps $\a,\b,\tilde{\a},\tilde{\b}$ commute. Define new multiplications on $Q$ by $p \swarrow'_\l q = \tilde{\a}(p) \swarrow_\l \tilde{\b}(q)$, $p \searrow'_\l q= \tilde{\a}(p) \searrow_\l \tilde{\b}(q),$ $p \nwarrow'_\l q= \tilde{\a}(p) \nwarrow_\l \tilde{\b}(q)$ and $p \nearrow'_\l q = \tilde{\a}(p)_\l \nearrow \tilde{\b}(q)$. Then $(Q,\nwarrow'_\l, \swarrow'_\l, \nearrow'_\l, \searrow'_\l, \a\circ\tilde{\a}, \b\circ \tilde{\b})$ is a BiHom-quadri conformal algebra.\end{rem}
\begin{rem}Let $(Q, \nwarrow_\l, \swarrow_\l, \nearrow_\l, \searrow_\l, \a, \b)$ be a BiHom-quadri-conformal algebra. By considering Eqs. (\ref{eq45})-(\ref{eq47}) and by column sum we obtain, for all $p, q, r \in Q$
	\begin{eqnarray}
	(p \<_\l q) \<_{\l+\m} \b(r) = \a(p) \<_\l (q *_\m r), \\(p \>_\l q) \<_{\l+\m} \b(r) = \a(p) \>_{\l} (q \<_{\m} r),\\ (p*_\l q) \>_{\l+\m} \b(r) = \a(p) \>_\l (q \>_\m r).
	\end{eqnarray}Thus, $(Q, \>_\l,\<_\l, \a, \b)$ is a BiHom-dendriform conformal algebra. By analogy with \cite{1}, we denote it by $Q_h$ and call it the horizontal BiHom-dendriform conformal algebra associated to $Q$.\end{rem}
Also by the row sum of Eqs.(\ref{eq45})-(\ref{eq47}) we obtain, for all $p, q, r \in Q:$ \begin{eqnarray}
(p \wedge_\l q) \wedge_{\l+\m}\b(r) = \a(p) \wedge_\l (q *_{\m} r),\\ (p\vee_\l q) \wedge_{\l+\m} \b(r) = \a(p) \vee_\l (q \wedge_{\l+\m} r),\\ (p *_\l q) \vee_{\l+\m} \b(r) = \a(p) \vee_\l (q \vee_{\m} r).
\end{eqnarray}Thus, $(Q, \wedge_\l,\vee_\l, \a, \b)$ is a BiHom-dendriform conformal algebra, which, again by analogy with \cite{1}, is denoted by $Q_v$ and is called the vertical BiHom-dendriform conformal algebra associated to $Q.$ From Corollary \ref{corbihomdend} we immediately obtain:
\begin{cor}
	Let $(Q, \nwarrow_\l, \swarrow_\l, \nearrow_\l, \searrow_\l, \a, \b)$ be a BiHom-quadri conformal algebra. Then $(Q, *_\l,\a,\b)$ is a BiHom-associative conformal algebra, where $p *_\l q = p \nwarrow_\l q+ p\swarrow_\l q+ p\nearrow_\l q +p\searrow_\l q$ for all $p, q\in Q$.
\end{cor}As seen below, the tensor product of two BiHom-dendriform conformal algebras naturally yields a BiHom-quadri conformal algebra.
\begin{prop}
	Consider two BiHom-dendriform conformal algebras $(A, \<_\l, \>_\l, \a, \b)$ and $(B, \<'_\l, \>'_\l, \a', \b')$. Define bilinear operations on the tensor product $ A \otimes B$ (for all $p_1, p_2 \in A$ and $q_1, q_2 \in B$) by:
	\begin{align*}
	&(p_1 \otimes q_1) \nwarrow_\l (p_2 \otimes q_2) = (p_1 \<_\l p_2) \otimes (q_1 \<'_\l q_2),\\
	&(p_1 \otimes q_1) \swarrow_\l (p_2 \otimes q_2) = (p_1 \<_\l p_2) \otimes (q_1 \>'_\l q_2),\\
	&(p_1 \otimes q_1) \nearrow_\l (p_2 \otimes q_2) = (p_1 \>_\l p_2) \otimes (q_1 \<'_\l q_2),\\
	&(p_1 \otimes q_1) \searrow_\l (p_2 \otimes q_2) = (p_1 \>_\l p_2). \otimes (q_1 \>'_\l q_2)
	\end{align*}
	Then $(A\otimes B , \nwarrow_\l, \swarrow_\l, \nearrow_\l, \searrow_\l, \a\otimes \a', \b\otimes\b')$ is a BiHom-quadri  conformal algebra.
\end{prop}
\begin{proof}The relations Eq.(\ref{eq43}) and Eq. (\ref{eq44}) are obvious. We denote by $\>_{\l}^{\otimes}, \<_{\l}^{\otimes},\vee_{\l}^{\otimes},\wedge_{\l}^{\otimes},*_{\l}^{\otimes}$ the operations defined on $A \otimes B$ by Eqs. (\ref{eq48})-(\ref{eq52}) corresponding to the operations $\swarrow_\l,\nwarrow_\l, \searrow_\l, \nearrow_\l$ defined above.
	One can easily see that, for all $p_1, p_2 \in A$ and $q_1, q_2 \in B$, we have
	\begin{align*}
	&(p_1 \otimes q_1) \<_{\l}^{\otimes} (p_2\otimes q_2) = (p_1 \<_\l p_2) \otimes (q_1 *' q_2),\\&
	(p_1 \otimes q_1) \>_{\l}^{\otimes} (p_2 \otimes q_2) = (p_1 \>_\l p_2) \otimes (q_1 *'_\l q_2),\\&
	(p_1 \otimes q_1) \wedge_{\l}^{\otimes}(p_2 \otimes q_2) = (p_1 *_\l p_2) \otimes (q_1 \<'_\l q_2),\\&
	(p_1 \otimes q_1) \vee_{\l}^{\otimes}(p_2 \otimes q_2) = (p_1 *_\l p_2) \otimes (q_1 \>'_\l q_2),\\&(p_1 \otimes q_1)*_{\l}^{\otimes}(p_2 \otimes q_2) = (p_1 *_\l p_2) \otimes (q_1 *'_\l q_2),\end{align*}
	Now we prove Eq.(\ref{eq45}) and leave the rest to the reader:
	\begin{equation*}
	\begin{aligned}
	((p_1 \otimes q_1)\nwarrow_\l  (p_2 \otimes q_2))\nwarrow_{\l+\m} (\b \otimes \b')(p_3 \otimes q_3)& = ((p_1\<_\l p_2) \otimes (q_1 \<'_\l q_2)) \nwarrow_{\l+\m} (\b(p_3) \otimes \b'(q_3))\\& = ((p_1 \<_\l p_2) \<_{\l+\m}\b(p_3)) \otimes((q_1 \<'_\l q_2) \<'_{\l+\m} \b' (q_3))\\&
	= (\a(p_1) \<_{\l}(p_2 *_\m p_3)) \otimes (\a'(q_1)\<'_\l(q_2 *'_{\m} q_3))\\&
	=(\a(p_1)\otimes\a'(q_1))\nwarrow_{\l}((p_2 *_\m p_3)\otimes (q_2 *'_{\m} q_3))\\&=(\a\otimes\a')(p_1\otimes q_1)\nwarrow_{\l}((p_2 \otimes q_2)*_{\m}^{\otimes} (p_3 \otimes q_3)).
	\end{aligned}
	\end{equation*}
	\begin{equation*}
	\begin{aligned}
((p_1 \otimes q_1)\nearrow_\l  (p_2 \otimes q_2))\nwarrow_{\l+\m} (\b \otimes \b')(p_3 \otimes q_3)& = ((p_1\>_\l p_2) \otimes (q_1 \<'_\l q_2)) \nearrow_{\l+\m} (\b(p_3) \otimes \b'(q_3))\\&
	= ((p_1 \>_\l p_2) \<_{\l+\m}\b(p_3)) \otimes((q_1 \<'_\l q_2) \<'_{\l+\m} \b' (q_3))\\&
	= (\a(p_1) \>_{\l}(p_2 \<_\m p_3)) \otimes (\a'(q_1)\<'_\l(q_2 *'_{\m} q_3))\\&
	=(\a(p_1)\otimes\a'(q_1))\nearrow_{\l}((p_2 \<_\m p_3)\otimes (q_2 *'_{\m} q_3))\\&=(\a\otimes\a')(p_1\otimes q_1)\nearrow_{\l}((p_2 \otimes q_2)\<_{\m}^{\otimes} (p_3 \otimes q_3)).
	\end{aligned}
	\end{equation*}
	\begin{equation*}
	\begin{aligned}
	((p_1\otimes q_1) \wedge_\l (p_2\otimes q_2)) \nearrow_{\l+\m} (\b\otimes\b')(p_3\otimes q_3)& = ((p_1*_\l p_2) \otimes (q_1\<'_\l q_2)) \nearrow_{\l+\m} (\b(p_3)\otimes\b'(q_3))\\&=((p_1*_\l p_2)\>_{\l+\m} \b(p_3))\otimes ((q_1\<'_\l q_2)\<'_{\l+\m} \b'(q_3))\\&=(\a(p_1) \>_\l (p_2 \>_{\m} p_3))\otimes (\a'(q_1) \<'_\l (q_2 *'_{\m} q_3))\\&=(\a(p_1) \otimes\a'(q_1))\nearrow_\l (p_2 \>_{\m} p_3)\otimes (q_2 *'_{\m} q_3)\\&=(\a\otimes\a')(p_1\otimes q_1)\nearrow_\l ((p_2 \otimes q_2) \>_{\l}^{\otimes} (p_3 \otimes q_3)).
	\end{aligned}
	\end{equation*} It completes the proof. 
\end{proof}Relation between BiHom-quadri conformal algebras and Rota-Baxter operators is given as follows.\begin{prop}\label{prop6.7}
Let $R: D \to D$ be a Rota-Baxter operator of weight $0$ and the algebra $(D,\<_\l, \>_\l, \a, \b)$ be a BiHom-dendriform conformal algebra. Define new operations on $D$ by using $p\searrow_{\l}^{R} q= R(p)\>_\l q$, $p \nearrow_{\l}^{R} q = p\>_\l R(q),$ $p \swarrow_{\l}^{R} q = R(p) \<_\l q$ and $p\nwarrow_{\l}^{R} q = p \<_\l R(q).$ Then,  $(D, \nwarrow_{\l}^{R}, \swarrow_{\l}^{R}, \nearrow_{\l}^{R}, \searrow_{\l}^{R}, \a, \b)$ is defined as a BiHom-quadri conformal algebra .
\end{prop}
\begin{proof}Here we check only one axiom from Def. \ref{BHQCA} and leave the rest to the reader. We denote by $\>_{\l}^{R}, ,\<_{\l}^{R},\vee_{\l}^{R},\wedge_{\l}^{R}, *_{\l}^{R}$ the operations defined on $D$ with respect to  $\nwarrow_{\l}^{R}, \swarrow_{\l}^{R}, \nearrow_{\l}^{R}, \searrow_{\l}^{R}$, that are defined as follows
\begin{equation*}
\begin{aligned} &p \<_{\l}^{R} q = p \nwarrow_{\l}^{R} q + p \swarrow_{\l}^{R} q = p\<_\l R(q) + R(p) \<_\l q,\\& p \>_{\l}^{R} q = p \searrow_{\l}^{R}q + p \nearrow_{\l}^{R} q = R(p)\>_\l q + p \>_\l R(q),\\&   p\wedge_{\l}^{R} q = p \nearrow_{\l}^{R} q + p \nwarrow_{\l}^{R} q = p \>_\l R(q) + p \<_\l R(q),\\& p\vee_{\l}^{R} q=p\swarrow_{\l}^{R} q+ p\searrow_{\l}^{R} q= R(p) \<_{\l} q + R(p) \>_\l q , \\&p *_{\l}^{R} q = p\nwarrow_{\l}^{R} q + p \swarrow_{\l}^{R} q + p \nearrow_{\l}^{R} q + p \searrow_{\l}^{R}q= p \<_\l R(q) + R(p) \<_\l q + R(p) \>_\l q + p \>_\l R(q).\end{aligned}\end{equation*}
Now we have \begin{equation*}
\begin{aligned}
(p \swarrow_{\l}^{R}q) \nwarrow_{\l+\m}^{R} \b(r) &= (R(p) \<_{\l} q) \<_{\l+\m} R(\b(r)) \\&= (R(p) \<_\l q) \<_{\l+\m} \b(R(r))
\\&\overset{(\ref{D3})}{=}
\a(R(p)) \<_\l (q \<_\m R(r) + q\>_\m R(r))
\\&= R(\a(p)) \<_\l (q \nwarrow_{\m}^{R}  r + q \nearrow_{\m}^{R} r)\\& = \a(p) \swarrow_{\l}^{R}(q \wedge_{\m}^{R} r).
\end{aligned}
\end{equation*}
And \begin{equation*}
\begin{aligned}
(p \wedge_{\l}^{R} q) \nearrow_{\l+\m} ^{R} \b(r) &= (p\nearrow_{\l}^{R} q +p \nwarrow_{\l}^{R} q)\>_{\l+\m} R(\b(r))
\\&= (p \<_\l R(q) + p \>_\l R(q)) \>_{\l+\m} \b(R(r))
\\ &\overset{(\ref{D5})}{=} \a(p) \>_\l (R(q) \>_\m R(r))
\\&\overset{(\ref{eq59})}{=}
 \a(p) \>_\l R(q \>_\m R(r) + R(q) \>_\m r)
\\&= \a(p)\>_\l R(q \nearrow_{\m}^{R} r + q \searrow_{\m}^{R} r) \\&= \a(p) \nearrow_{\l}^{R} (q \>_{\m}^{R} r).\end{aligned}
\end{equation*}As required.\end{proof}
\begin{rem}
In the setting of Proposition \ref{prop6.7}, the axioms of the Definition \ref{def6.6} can be rewritten as $R(p \>_{\l}^{R} q) =R(p) \>_\l R(q) $ and $ R(p \<_{\l}^{R} q) = R(p) \<_\l R(q). $ Thus, $R$ is a morphism of BiHom-dendriform conformal algebras from $D_h = (D, \<_{\l}^{R}, \>_{\l}^{R}, \a,\b)$ to $(D, \<_\l, \>_\l, \a,\b)$.\end{rem}
Contrarily, if we denote by $(D, *_\l, \a,\b)$ the BiHom-associative conformal algebra generated from $D$ as in Corollary \ref{corbihomdend}, it is apparent that we have $p \wedge_{\l}^{R} q = p *_{\l} R(q)$ and $p \vee_{\l}^{R} q = R(p) *_\l q,$ for all $p, q\in D.$ In this context, the BiHom-dendriform conformal algebra structure established on $D$ by using Corollary \ref{cor6.3} (for the Rota-Baxter operator $R$ on the BiHom-associative conformal algebra $(D,*_\l,\a,\b)$) is precisely the vertical BiHom-dendriform conformal algebra $D_v= (D, \wedge_{\l}^{R}, \vee_{\l}^{R}, \a, \b)$. And this $D_v$ is obtained from the BiHom-quadri- conformal algebra $(D,\nwarrow_{\l}^{R},\swarrow_{\l}^{R}, \nearrow_{\l}^{R},\searrow_{\l}^{R},\a,\b)$.

 Moreover, pair of commuting Rota-Baxter operators on BiHom-associative conformal algebras provides examples of BiHom-quadri conformal algebras, as in \cite{1}.
\begin{prop}
	On the BiHom-associative algebra  $(A, \tau_\l, \a, \b)$, let $R$ and $P$ be a pair of commuting Rota-Baxter operators of weight $0$ such that $ \a\circ R=R \circ \a $ and $ \b\circ R =R\circ \b$, $\a\circ P= P \circ \a$ and $ \b\circ P= P\circ \b$. Then $P$ is a Rota-Baxter operator with weight $0$ on the BiHom-dendriform conformal algebra $(A, \<_{\l}^{R}, \>_{\l}^{R}, \a, \b)$ corresponding to $R$, as shown in Corollary \ref{cor6.3}.
\end{prop}\begin{proof}We verify the axioms in equations (cf. \ref{eq59} and \ref{eq60}). For all $p, q \in A$, we have:
\begin{equation*}
\begin{aligned}
P(p) \<_{\l}^{R} P(q) = P(p)_{\l} R(P(q)) &= P(p)_\l P(R(q))
\\&\overset{(\ref{eq1})}{=} P(p_\l P(R(q)) + P(p)_\l R(q))
\\&= P(p_\l R(P(q)) + P(p)_\l R(q))
\\&= P(p \<_{\l}^{R} P(q) + P(p) \<_{\l}^{R} q).
\end{aligned}
\end{equation*}And\begin{equation*}
\begin{aligned}
P(p)\>_{\l}^{R} P(q) = R(P(p))_\l P(q) &= P(R(p))_\l P(q)
\\&\overset{(\ref{eq1})}{=} P(R(p)_\l P(q) + P(R(p))_\l q) \\&= P(R(p)_\l P(q) + R(P(p))_\l q)
\\&= P(p\>_{\l}^{R} P(q) + P(p) \>_{\l}^{R} q).
\end{aligned}
\end{equation*}
As required.
\end{proof}
\section{ Rota-Baxter system on the BiHom-associative conformal algebras}
In this section, we define a generalization of Rota-Baxter operators that is called as Rota-Baxter system in the BiHom-associative conformal algebra and BiHom-dendriform conformal algebras' context. We further investigate the way to establish a BiHom-quadri-conformal algebra from Rota-Baxter system on BiHom-dendriform conformal algebras. A Rota-Baxter system can be defined as follows:
\begin{defn}Assuming that $A$ be an associative conformal algebra with  two $\mathbb{C}[\p]$- module morphisms $R, S : A \to A$, then a Rota-Baxter system is the triple $(A, R, S)$, if for all $p, q \in A$ and $\l\in \mathbb{C}$, following conditions hold \begin{eqnarray}
	R(p)_\l R(q) &= R(R(p)_\l q + p_\l S(q)), \\
	S(p)_{\l}S(q) &= S(R(p)_\l q + p_\l S(q)).
	\end{eqnarray}
\end{defn}
\begin{defn}\label{def7.4} Assume that $(A, \tau_\l, \a, \b)$ is a BiHom-associative conformal algebra with two $\mathbb{C}[\p]$-module morphisms $R, S : A \to A$. A Rota-Baxter system  $(A, \a, \b, R, S)$ on BiHom associative conformal algebra satisfies,\begin{eqnarray}
\a R = R \a, R \b = \b R, & \a S = S \a, S \b = \b S, \label{1122} \\
R(p)_\l R(q) &= R(R(p)_\l q + p_\l S(q)), \label{1133}\\
S(p)_\l S(q) &= S(R(p)_\l q + p_\l S(q)),\label{1144} 	\end{eqnarray}for all $p, q\in A$ and $\l\in \mathbb{C}$.\end{defn}
\begin{prop} 
	Consider an associative conformal algebra $(A, \tau_\l)$ and two commuting algebra endomorphisms $\a, \b: A \to A$. In order for $(A, R, S)$ to be a Rota-Baxter system on associative conformal algebra $A$, it is assumed that $R, S: A \to A$ are $\mathbb{C}[\p]$-module morphisms commuting with $\a, \b$. If we introduce a new multiplication on $A$ by $p *_\l q = \a(p)_\l \b(q)$ for all $p, q \in A$, then $A_{\a,\b} = (A, *_\l, \a, \b)$ is a BiHom-associative conformal algebra, also known as the Yau twist of $(A, \tau_\l)$, and $(A_{\a, \b}, \a,\b, R, S)$ also forms a Rota-Baxter system.\end{prop}
\begin{proof}One by one, we prove each of the conditions in definition \ref{def7.4}:
	\begin{equation*}\begin{aligned}
	R(p)*_\l R(q)=& \a(R(p))_\l\b(R(q)) \\&= R(\a(p))_\l R(\b(q))
	\\&= R(R(\a(p))_\l \b(q) + \a(p)_\l S(\b(q)))
	\\&= R(\a(R(p))_\l \b(q) + \a(p)_\l \b(S(q)))
	\\&= R(R(p) *_\l q + p *_\l S(q)).
	\end{aligned}
	\end{equation*}
	\begin{equation*}\begin{aligned}
	S(p)*_\l S(q)=& \a(S(p))_\l\b(S(q)) \\&= S(\a(p))_\l S(\b(q))
	\\&= S(S(\a(p))_\l \b(q) + \a(p)_\l R(\b(q)))
	\\&= S(\a(S(p))_\l \b(q) + \a(p)_\l \b(R(q)))
	\\&= S(S(p) *_\l q + p *_\l R(q)).
	\end{aligned}
	\end{equation*}It completes the proof.
\end{proof}Rota-Baxter operators in Eq. (\ref{eq1}) are examples of Rota-Baxter systems as stated in the following assertion.	
\begin{prop}Assume that $(A, \tau_\l,\a,\b)$ is a BiHom-associative conformal algebra and that $R: A \to A$ is a Rota-Baxter operator of weight $\theta$ that commutes with $\a$ and $\b$. In that case, $(A, \a, \b, R, R + \theta id)$ and $(A, \a, \b, R + \theta id, R)$ are Rota-Baxter systems, where $R + \theta id$ is a map from $A$ to $A$, defined by $(R + \theta id)(p) = R(p) + \theta p$ for all $p \in A$.
\end{prop}\begin{proof}We just establish that the system $(A, \a, \b, R, R+  \theta id)$ is a Rota-Baxter system; and the other case is analogous. It is easy to demonstrate Eq. (\ref{1122}). We calculate, Eq. (\ref{1133}) as follows
	\begin{equation*}\begin{aligned}
	R(p)_\l R(q)&= R(R(p)_\l q + p_\l R(q) + \theta (p_\l q))\\&= R(R(p)_\l q + p_\l (R + \theta id)(q)).
	\end{aligned}
	\end{equation*}
	For Eq. (\ref{1144}), we compute:\begin{equation*}\begin{aligned}
	&(R + \theta id)(p)_\l(R + \theta id)(q)
	\\&= (R(p) + \theta p)_\l(R(q) + \theta q) = R(p)_\l R(q) + \theta R(p)_\l q +  \theta p_\l R(q) + \theta^2(p_\l q)
	\\& = R(R(p)_\l q + p_\l R(q) + \theta(p_\l q)) + \theta(R(p)_\l q + p_\l R(q) + \theta (p_\l q))
	\\& = (R + \theta id)(R(p)_\l q + p_\l R(q) + \theta(p_\l q))\\& = (R + \theta id)(R(p)_\l q + p_\l (R + \theta id)(q)).
	\end{aligned}
	\end{equation*}It completes the proof.
\end{proof}

Following is how the Rota-Baxter systems are related to BiHom-dendriform confomal algebras:
 \begin{thm}\label{thm1}
 	Assume that $(A, \tau_\l, \a,\b)$ is a BiHom-associative conformal algebra and that $R, S: A \to A$ be a $\mathbb{C}$-linear maps that commute with $\a$ and $\b$. For any $p, q in A,$ define the operations $\<_\l$ and $\>_\l$ as $p \<_\l q = p_\l S(q)$ and $p \>_\l q = R(p)_\l q, $ respectively. In the case if $(A,\a,\b, R, S)$ is a Rota-Baxter system, then $(A,\<_\l,\>_\l,\a,\b)$ is a BiHom-dendriform conformal algebra.
 \end{thm}
\begin{proof} We only prove Eqs. (\ref{D3})-(\ref{D5}) and leave the rest to the reader. We have: \begin{equation*}
		\begin{aligned}
		&(p \<_\l q) \<_{\l+\m} \b(r) - \a(p) \<_\l(q \<_\m r) - \a(p)\<_\l(q \>_\m r)
		\\&= (p_\l S(q))\<_{\l+\m}\b(r) - \a(p) \<_{\l} (q_\m S(r)) - \a(p) \<_\l (R(q)_\m r)\\&
		= (p_\l S(q))_{\l+\m}S(\b(r)) -\a(p)_{\l}S(q_\m S(r)) -\a(p)_\l S(R(q)_\m r)\\&
		= (p_\l S(q))_{\l+\m}\b(S(r)) -\a(p)_\l S(R(q)_\m r + q_\m S(r))\\&
	= \a(p)_\l(S(q)_\m S(r)) -\a(p)_\l S(R(q)_\m r + q_\m S(r))\\&
		= \a(p)_\l(S(q)_\m S(r)) -\a(p)_\l (S(q)_\m S(r))\\& = 0.
		\end{aligned}
		\end{equation*}
		\begin{equation*}
		\begin{aligned}
		&(p \>_\l q)\<_{\l+\m} \b(r) - \a(p) \>_\l (q \<_{\m} r)\\&
		= (R(p)_\l q) \<_{\l+\m} \b(r) -\a(p) \>_\l (q_\m S(r))\\&
		= (R(p)_\l q)_{\l+\m}S(\b(r)) - R(\a(p))_\l(q_\m S(r))\\&
		= (R(p)_\l q)_{\l+\m}\b(S(r)) - \a(R(p))_\l(q_\m S(r))\\&
= \a(R(p))_\l(q_\m S(r)) -\a (R(p))_\l(q_\m S(r)) \\&= 0.
		\end{aligned}
		\end{equation*}
		And \begin{equation*}
		\begin{aligned}
		&\a(p) \>_\l (q \>_\m r) - (p \<_\l q) \>_{\l+\m} \b(r) - (p \>_\l q) \>_{\l+\m} \b(r)\\&
		= \a(p) \>_\l(R(q)_\m r) - (p_\l S(q)) \>_{\l+\m}\b(r) -(R(p)_\l q) \>_{\l+\m}\b(r)\\&
		= R(\a(p))_\l(R(q)_\m r)- R(p_\l S(q))_{\l+\m}\b(r)- R(R(p)_{\l}q)_{\l+\m}\b(r)\\&
		= \a(R(p))_\l(R(q)_\m r)- (R(p)_\l R(q))_{\l+\m}\b(r)\\&
		= (R(p)_\l R(q))_{\l+\m}\b(r) - (R(p)_\l R(q))_{\l+\m}\b(r) \\&= 0.
		\end{aligned}
		\end{equation*}It completes the proof.	
\end{proof}
We have a few unique instances of Theorem \ref{thm1}.
 \begin{cor}Assume that $R, S: A \to A$ as described above and that $A$ is an associative conformal algebra. Then, for every $p, q \in A$, $p \<_\l q = p_\l S(q)$ and $p \>_\l q = R(p)_\l q,$ we may define new operations $ \<_\l$ and $\>_\l$ on $A$. Assuming that $(A, R, S)$ is a Rota-Baxter system, $(A, \<_\l, \>_\l)$ is a dendriform conformal algebra.
\end{cor} 
\begin{proof}By assuming $\a=\b = id_A,$ in Theorem \ref{thm1}.
	\end{proof}
\begin{cor}Let $R: A \to A$ be a Rota-Baxter operator of weight $0$ commuting with $a$ and $b$ and $(A,\tau_\l,\a,\b)$ be a BiHom-associative conformal algebra. For any $p, q \in A$, define the operations $\<_\l$ and $\>_\l$ by 	$p \<_\l q = p_\l R(q)$ and $p\>_\l q = R(p)_\l q$ respectively. Then $(A, \<_\l,\>_\l,\a,\b)$ is a BiHom-dendriform conformal algebra. 
\end{cor} 
\begin{proof}By employing $R = S$ in the Theorem \ref{thm1}, it is obvious that $R$ is a Rota-Baxter operator with weight $0$.
	\end{proof}
 As the Rota-Baxter systems for the BiHom-associative conformal algebras, we now define the Rota-Baxter system for the BiHom-dendriform conformal algebra in the following manner.
\begin{defn}
	Let $(D,\<_\l,\>_\l,\a,\b)$ be a BiHom-dendriform conformal algebra and $R, S : A\to A$  are two
	$\mathbb{C}$-linear operators. We call $(D,\<_\l,\>_\l,\a,\b, R, S)$ a Rota–Baxter system on a BiHom dendriform conformal algebra if $R \a=\a R $
	$R\b=\b R$, $S\a=\a S$, $S \b = \b S$ and the following conditions are satisfied, for all $p, q \in D$:
	\begin{equation}\label{eq80}
	\begin{aligned}
	R(p) \>_\l R(q)& = R(p\>_\l S(q) + R(p)\>_\l q), \\
	R(p) \<_\l R(q)& = R(p \<_\l S(q) + R(p) \<_\l q), \\
	S(p) \>_\l S(q) &= S(p \>_\l S(q) + R(p)\>_\l q),\\
	S(p) \<_\l S(q)& = S(p \<_\l S(q) + R(p) \<_\l q). 
	\end{aligned}
	\end{equation}
\end{defn}
 This construction will further aid us to construct BiHom-quadri conformal algebra for Rota-Baxter system on dendriform conformal algebra.
\begin{prop}\label{prop7.6}
	Consider $(D, \<_\l, \>_\l, \a, \b)$ and $(D, R, S)$ are a BiHom-dendriform conformal algebra and a Rota-Baxter
	system respectively. Define new operations on $D$ as follows: $$p \swarrow_{\l}^{R} q = R(p) \>_\l q, p \nwarrow_{\l}^{R}q = p \>_\l S(q), p \searrow_{\l}^{R} q =
	R(p) \<_\l q \textit{ and }p \nearrow_{\l}^{R} q = p \<_\l S(q),\textit{ for all } p, q\in D.$$ A BiHom-quadri conformal algebra is then defined as $(D, \nwarrow_{\l}^{R}, \swarrow_{\l}^{R}, \nearrow_{\l}^{R}, \searrow_{\l}^{R}, \a,\b)$.
\end{prop}
\begin{proof}The operations $\>_{\l}^{R}, \<_{\l}^{R}, \vee_{\l}^{R}, \wedge_{\l}^{R},*_{\l}^{R}$ defined on $D$ by Eqs. (\ref{eq48})-(\ref{eq52}) corresponding to the operations $\nwarrow_{\l}^{R}, \swarrow_{\l}^{R}, \nearrow_{\l}^{R}, \searrow_{\l}^{R}$ are
	\begin{equation*}
	\begin{aligned}
	p \<_{\l}^{R} q  &:= p \nwarrow_{\l}^{R} q + p \swarrow_{\l}^{R} q= p\<_\l S(q)+ R(p)\<_\l q,\\ p \>_{\l}^{R} q&:= p \nearrow_{\l}^{R} q + p \searrow_{\l}^{R} q= p\>_{\l}S(q)+ R(p)\>_\l q,\\ p \vee _{\l}^{R} q &:= p \searrow_{\l}^{R} q + p \swarrow_{\l}^{R} q= p\>_\l S(q)+ p\<_\l S(q),\\ p \wedge_{\l}^{R} q &:= p \nearrow_{\l}^{R} q + p \nwarrow_{\l}^{R} q= R(p)\<_\l q+R(p)\>_\l q,\\
	p*_{\l}^{R} q &:= p \nwarrow_{\l}^{R} q + p \swarrow_{\l}^{R} q + p \searrow_{\l}^{R} q + p \nearrow_{\l}^{R} q= p\<_\l S(q)+R(p)\<_\l q+ p\>_\l S(q)+R(p)\>_\l q
	\end{aligned}
	\end{equation*}for all $p, q \in D$. We only check the conditions in Eq. (\ref{eq45}) as follows:
	\begin{equation*}
	\begin{aligned}
	\a(p)\nwarrow_{\l}^{R}(q*_{\m}^{R} r)=&\a(p)\<_\l S(q*_{\m}^{R} r)\\&=\a(p)\<_\l S(q\<_\m S(r)+R(q)\<_\m r+ q\>_\m S(r)+R(q)\>_\m r)\\&=\a(p)\<_\l S(q\<_\m S(r)+R(q)\<_\m r)+ \a(p)\<_\l S(q\>_\m S(r)+R(q)\>_\m r)\\&=\a(p)\<_\l (S(q)\>_\m S(r)+ S(q)\<_\m S(r))\\&=(p\<_\l S(q))\<_{\l+\m}\b(S(r))\\&=(p\<_\l S(q))\<_{\l+\m}S(\b(r))\\&=(p\nwarrow_{\l}^{R}S(q))\nwarrow_{\l+\m}^{R}S(\b(r)),
	\end{aligned}
	\end{equation*}
	\begin{equation*}
	\begin{aligned}
	\a(p)\nearrow_{\l}^{R}(q\<_{\m}^{R} r)&=\a(p)\nearrow_{\l}^{R}(q\<_\m S(r)+ R(q)\<_\m r)\\&=\a(p)\>_{\l}S(q\<_\m S(r)+ R(q)\<_\m r)\\&=\a(p)\>_{\l}(S(q)\<_\m S(r))\\&=(p\>_{\l}S(q))\<_\m \b S(r)\\&=(p\>_{\l}S(q))\<_\m  S\b(r)\\&=(p\nearrow_{\l}^{R}S(q))\nwarrow_{\l+\m}^{R} \b S(r),
	\end{aligned}
	\end{equation*}
And	\begin{equation*} \begin{aligned} (p \wedge_{\l}^{R} q) \nearrow_{\l+\m }^{R} \b(r) =& (p \>_\l S(q) + p \<_\l S(q)) \>_{\l+\m} S(\b(r))\\&
	= (p \>_\l S(q) + p \<_\l S(q)) \>_{\l+\m} \b(S(r))\\&
	= \a(p) \>_\l (S(q) \>_{\m}S(r))\\&
	= \a(p) \>_\l S(R(q)\>_\m r + q \>_\m S(r))\\&
	= \a(p) \nearrow_{\l}^{R}(R(q) \>_{\m}r + q \>_\m S(r))\\&
	=\a(p) \nearrow_{\l}^{R} (q \>_{\m}^{R} r).\end{aligned}\end{equation*}
	However, other conditions can be prove similarly.
\end{proof}
\begin{rem}
	In the setting of Proposition \ref{prop7.6}, the axioms in Eq. (\ref{eq80}) can be rewritten as $ R(p \>_{\l}^{R} q) = R(p)\>_\l R(q)$, $ R(p \<_{\l}^{R}q)= R(p) \<_\l R(q)$ and $ S(p \>_{\l}^{R} q)= S(p) \>_\l S(q)$, $ S(p \<_{\l}^{R} q)= S(p) \<_\l S(q) $. Thus, $R$ and $S$ are morphisms from $D_h = (D, \<_{\l}^{R} , \>_{\l}^{R}, \a,\b)$ to $(D, \<_\l, \>_\l, \a, \b)$.
\end{rem}However, if we represent the BiHom-associative conformal algebra generated from $D$ as in the Corollary \ref{corbihomdend}, it is clear that we have $p \wedge_{\l}^{R} q = p *_\l S(q)$ and $p \vee_{\l}^{R} q = R(p) *_\l q,$ for all $p,q\in D$. Thus, the BiHom-dendriform conformal algebra structure acquired on $D$ by applying Theorem \ref{thm1} for the Rota-Baxter system $(R, S)$ on the BiHom-associative conformal algebra $(D,*_\l,\a,\b)$ is identical to the vertical BiHom-dendriform conformal algebra $D_v = (D, \wedge_{\l}^{R}, \vee_{\l}^{R}, \a,\b)$.
\begin{prop}\label{prop7.7}Let  $(A, \a, \b, R, S)$ and $(A, \a,\b , R',S')$  are two Rota-Baxter systems on a BiHom-associative algebra $(A, \tau_\l, \a, \b)$, where $S S'= S' S$,  $ R S' = S'R, $ SR'= R' S$, $ $R R'=R' R$. Then $(R', S')$ is a Rota-Baxter system on the BiHom-dendriform conformal algebra $(A, \<_{\l}^{R},\>_{\l}^{R},\a,\b)$, where $p \<_{\l}^{R} q = p_\l S(q)$ and $p \>_{\l}^{R} q = R(p)_\l q$ for all $p,q\in A.$\end{prop}
\begin{proof}By checking the axioms in Eq.(\ref{eq80}) one by one, we have  
	\begin{equation*}
	\begin{aligned}
	R'(p) \>_{\l}^{R} R' (q) =& R(R' (p))_\l R' (q) \\&= R' (R(p))_\l R' (q)\\&= R' (R' (R(p))_\l q + R(p)_\l S'(q)) \\&= R' (R(R' (p))_\l q + R(p)_\l S'(q))
	\\&= R' (R' (p)\>_{\l}^{R} q + p \>_{\l}^{R} S' (q)),
	\end{aligned}
	\end{equation*}
	\begin{equation*}
	\begin{aligned}R' (p) \<_{\l}^{R} R' (q) &= R'(p)_\l S(R' (q))
	\\&= R' (p)_\l R' (S(q))
	\\&= R'(R' (p)_\l S(q) + p_\l S' (S(q)))
	\\& = R' (R' (p)_\l S(q) + p_\l S(S' (q)))
	\\&= R'(R' (p) \<_{\l}^{R}q + p\<_{\l}^{R} S'(q)),
	\end{aligned}
	\end{equation*}\begin{equation*}
	\begin{aligned}S'(p)\>_{\l}^{R}S'(q)&= R(S'(p))_\l S'(q)\\&= S'(R(p))_\l S'(q)\\&= S'(R'(R(p)_\l q + R(p)_\l S'(q))\\&= S'(R(R'(p))_\l q + R(p)_\l S'(q))\\&= S'(R'(p)\>_{\l}^{R} q+ p\>_{\l}^{R} S' (q)),\end{aligned}
	\end{equation*}and\begin{equation*}\begin{aligned}S'(p)\<_{\l}^{R}S'(q) &= S'(p)_\l S(S'(q))\\&= S'(p)_\l S'(S(q))\\&= S'(R'(p)_\l S(q) + p_\l S'(S(q)))\\&= S'(R'(p)_\l S(q) + p_\l S(S'(q)))\\&=S'(R'(p)\<_{\l}^{R}q + p \<_{\l R}S'(q)),\end{aligned}\end{equation*}for all $p,q\in  A$. It completes the proof.
\end{proof}Applying Proposition \ref{prop7.6} to the Rota-Baxter system $(R ',S')$, given in the Proposition \ref{prop7.7}, we obtain the following result.
\begin{cor}
	There is a BiHom-quadri conformal algebra $(A, \nwarrow_{\l}^{R}, \swarrow_{\l}^{R}, \nearrow_{\l}^{R}, \searrow_{\l}^{R}, \a,\b)$ in the context of Proposition \ref{prop7.7}, where the operations are described by
	\begin{equation*}
	\begin{aligned}
   p \searrow_{\l}^{R} q =& R'(p) \>_{\l}^{R} q = R(R'(p))_\l q,\\ p \nearrow_\l q =& p \>_{\l}^{R} S'(q) = R(p)_\l S'(q),\\ p \swarrow_{\l}^{R} q =& R' (p) \<_{\l}^{R} q = R' (p)_\l S(q),\\ p \nwarrow_{\l}^{R} q =& p \<_{\l}^{R} S'(q) = p_\l S(S'(q)),
	\end{aligned}
	\end{equation*}
	for all $p,q \in A$.
\end{cor}
In particular, there is a new BiHom-associative conformal algebra $(A, *_\l, \a,\b)$, with $p *_\l q = R(R'(p))_\l q+ R(p)_\l S'(q)+ R'(p)_\l S(q)+ p_\l S(S'(q))$, for all $p,q \in A.$
\section*{Declarations}
\subsection*{Competing interests:}
The authors have no competing interests to declare that are relevant to the content of this paper.
\subsection*{Authors' contributions:} 
All authors contributed equally within the manuscript.
\subsection*{Funding:}This work is supported by the Jiangsu Natural Science Foundation Project (Natural Science Foundation of Jiangsu Province), Relative Gorenstein cotorsion Homology Theory and Its Applications (No.BK20181406).
\subsection*{Availability of data and materials:}
All data is available within the manuscript.
\end{document}